\documentclass[12pt,reqno]{amsart}
\pdfoutput=1
\usepackage{etoolbox} 
\patchcmd{\subsubsection}{\normalfont}{\bfseries\textnormal}{}{} 

\usepackage[vmargin=2cm,hmargin=2cm]{geometry} 

\usepackage{amssymb, latexsym, color, etaremune, enumerate, tikz}
\usepackage{caption,float,textcomp,units,xfrac,url,tabularx}

\usetikzlibrary{decorations.markings} 
\title[Landau and Ramanujan approximations for divisor sums and cusp forms]{Landau and Ramanujan approximations\\ for divisor sums and coefficients of cusp forms}

\author[A. Ciolan]{Alexandru Ciolan}
\address{Max-Planck-Institut f\"ur Mathematik, Vivatsgasse 7, D-53111 Bonn, Germany}
\email{ciolan@mpim-bonn.mpg.de}

\author[A. Languasco]{Alessandro Languasco}
\address{Universit\`a di Padova, Dipartimento di Matematica ``Tullio Levi-Civita'', via Trieste 63, 35121\\\phantom{11}\,Padova, Italy}
\email{alessandro.languasco@unipd.it}

\author[P. Moree]{Pieter Moree}
\address{Max-Planck-Institut f\"ur Mathematik, Vivatsgasse 7, D-53111 Bonn, Germany}
\email{moree@mpim-bonn.mpg.de}

\subjclass[2020]{Primary 11N37, 11F33; secondary 11Y60} 
\keywords{Divisor sums, cusp forms, congruences, tau-function, Landau and Ramanujan approximations,
Euler-Kronecker constants}

\usepackage{microtype} 
\usepackage[english]{babel} 
\newtheorem{Thm}{Theorem}
\newtheorem{claim}{Claim}
\newtheorem{Lem}{Lemma}
\newtheorem{CThm}{Classical Theorem}

\newtheorem{Cor}{Corollary}
\newtheorem{Con}{Conjecture}

\theoremstyle{definition}
\newtheorem{Rem}{Remark}
\newtheorem{Obser}{Numerical Observation}
\newtheorem{Def}{Definition}

\newtheorem{Prop}{Proposition}

\DeclareMathOperator{\ord}{ord}

\makeatletter
\let\@@pmod\pmod
\DeclareRobustCommand{\pmod}{\@ifstar\@pmods\@@pmod}
\def\@pmods#1{\mkern4mu({\operator@font mod}\mkern 6mu#1)}
\makeatother

\makeatletter 
\newcommand*\kronecker[2]{%
  \relax\if@display
    \expandafter{(\frac{#1}{#2})}
  \else
    \expandafter{\bigl(\frac{#1}{#2}\bigr)}
  \fi
}
\makeatother
\newcommand\lrkronecker[2]{\Bigl( \frac{#1}{#2} \Bigr)} 

\allowdisplaybreaks

\begin{document}

\begin{abstract} 
In 1961, Rankin  determined the asymptotic behavior of the number $S_{k,q}(x)$
of positive integers $n\le x$ for which a given prime $q$ does not divide $\sigma_k(n),$ the $k$-th divisor sum function.
By computing the associated Euler-Kronecker constant $\gamma_{k,q},$ which depends on the arithmetic of certain subfields of $\mathbb Q(\zeta_q)$,
 we obtain the second order term in the asymptotic expansion of $S_{k,q}(x).$
Using a method developed by Ford, Luca and Moree (2014), 
we determine the pairs $(k,q)$ with
$(k, q-1)=1$  for which Ramanujan's
approximation to $S_{k,q}(x)$  is better than Landau's. 
This entails checking
whether $\gamma_{k,q}<1/2$ or not, and requires a substantial computational number theoretic input and extensive computer usage.
We apply our results to study the non-divisibility of Fourier coefficients of six
cusp forms by certain exceptional primes, extending the earlier work of Moree (2004), who disproved several
claims made by Ramanujan on the non-divisibility of 
the Ramanujan tau function by five such  exceptional primes. 
\end{abstract}

\maketitle 

\section{Introduction}\label{Sec:Intro}
\subsection{Motivation and historical background}\label{subsec:Motivation}
A set $S$ of positive integers is said to be \emph{multiplicative} if for every pair $(m,n)$ of coprime 
positive integers we have
that $mn\in S$ if and only if  $m,n \in S$.  (In other words, $S$ is a multiplicative set if
and only if the indicator function of $S$ is multiplicative.) 
An enormous supply of multiplicative sets is provided 
by taking 
\begin{equation}
\label{Sdef}
S=\{n\ge 1:q\nmid f(n)\},
\end{equation}
where $f$ is a multiplicative function and 
$q$ a prime. (Throughout the paper, the letters $p$  and $q$ will always denote
prime numbers.) Several papers (see, e.g., \cite{FordLucaMoree,Narkiewicz,Rankin61,Scourfield64,
Scourfield70,Scourfield72,serre,SW}) 
are concerned with the asymptotic behavior of $S(x)$,  
the number of positive integers $n\le x$ that are in $S$.
An important role in understanding this quantity is played by the Dirichlet series 
\begin{equation}
\label{LS}
L_S(s) := \sum_{n\in S}n^{-s},
\end{equation}
which converges for $\Re(s) > 1$. 

Here we are interested in the second order behavior of $S(x)$ and, in particular, in the case
where $S=\{n\ge 1:q\nmid \sigma_k(n)\}$, with $\sigma_k(n)=\sum_{d\mid n}d^k$ being the usual  $k$-th divisor sum function. 
Our results have applications to the
non-divisibility of the Fourier coefficients
of six standard cusp forms by so-called \textit{exceptional primes}. The cusp forms that make the object of our study are the normalized generators of the six one-dimensional  cusp form spaces for the full modular group $\textrm{SL}_2({\mathbb Z})$ (see 
Table \ref{tab:formslist}). Of these,  the \textit{modular discriminant} function 
$$\Delta(z)=q_1\prod_{m=1}^{\infty}(1-q_1^m)^{24}=\sum_{n=1}^{\infty}\tau(n)q_1^n,$$
is perhaps the most well-known (with $z\in \mathbb H,$ the complex upper half-plane, and 
$q_1=e^{2\pi iz}$), its Fourier coefficients $\tau(n)$ being  the values of the 
\emph{Ramanujan tau function}. 
\par Ramanujan 
was not the first to consider $\Delta$, but he seems to have been the first to realize that the values of
$\tau(n)$ provide an interesting arithmetic sequence. In
an ``unpublished" 
manuscript that belongs to the collection
of Trinity College, Cambridge, he considered $\tau$ modulo various prime powers $q^e$ with $q\in \{2,3,5,7,23,691\}.$ Except for the case $q=23$, these congruences involve the divisor sum function (and, often, a power of $n$), the most
famous example in this regard being
$\tau(n)\equiv \sigma_{11}(n)\pmod*{691}$.
\par In 2004 the second order
behavior for $f(n)=\tau(n)$ and $q\in\{3,5,7,23,691\}$ 
was determined by Moree \cite{Ramadisproof}.
One aim of this paper is to put his work in a general framework.
First, in that we consider \eqref{Sdef} with $f=\sigma_{k}$ being \emph{any} sum of divisors function  and
$q$ \emph{any} prime. Second, in that we consider an entire class of cusp forms that share certain properties (e.g., they are normalized simultaneous Hecke eigenforms), out of which $\Delta$ is a representative.
\par The congruences found by Ramanujan for $q\in\{3,5,7,23,691\}$ are not singular, and certainly not a coincidence. The monumental work of Serre \cite{serre69,serre73a} 
and Swinnerton-Dyer \cite{S-D-ladic,S-D} revealed that these primes  are only a few out of a much larger, but finite, list of \textit{exceptional primes} modulo which the coefficients of these six cusp satisfy congruences involving divisor sum 
functions, as shown in 
Table \ref{tab:formslist}. 
\subsection{Euler-Kronecker constants} 
In the following we will use $F'/F(s)$ as a 
shorthand for $F'(s)/F(s)$.
If the limit
\begin{equation}
\label{EKf} 
\gamma_S:=\lim_{s\rightarrow 1^+}\bigg(
\frac{L'_S}{L_S}(s) 
+\frac{\alpha}{s-1}\bigg)
\end{equation}
exists for some $\alpha>0$, we say that the set $S$ admits an \emph{Euler-Kronecker constant} $\gamma_S$. In case
$S=\mathbb N,$ we have 
$L_S(s)=\zeta(s)$, the 
\emph{Riemann zeta function}, $\alpha=1$ and $\gamma_S=\gamma$, the \emph{Euler-Mascheroni 
constant} (see, for
example, 
Lagarias \cite{Lagarias} for a beautiful  survey, and 
Havil \cite{Havil} for a popular account).

As
the following result shows, the Euler-Kronecker constant $\gamma_S$ determines the second order behavior of $S(x),$ 
\begin{CThm}
\label{thm:multiplicativeset}
Let $S$ be a multiplicative set.
If there are $\rho>0$ and $0<\delta<1$ such that
\begin{equation}
\label{primecondition}
\sum_{p\le x,~p\not\in S}1=\delta~\sum_{p\le x}1+O_S\bigg(\frac{x}{\log^{2+\rho}x}\bigg),
\end{equation}
then $\gamma_S\in\mathbb R$ exists and 
\begin{equation}
\label{initstarrie}
S(x)=\sum_{n\le x,~n\in S}1=\frac{c_0\,x}{\log^{\delta}x}\bigg(1+\frac{(1-\gamma_S)\delta}{\log x}(1+o_S(1))\bigg)
\end{equation}
as $x\to\infty,$ where $c_0$ is a positive constant.
If  the prime numbers belonging to $S$ are, with finitely many exceptions, precisely
those in a finite union of arithmetic progressions, 
we have, for arbitrary $j\ge 1$,
\begin{equation}
\label{starrie}
S(x)=\frac{c_0\,x}{\log^{\delta}x}\left(1+\frac{c_1}{\log x}+\frac{c_2}{\log^2 x}+\dots+
\frac{c_j}{\log^j x}+O_{j,S}\left(\frac{1}{\log^{j+1}x}\right)\right),
\end{equation}
with $c_0,\ldots,c_j$ constants, $c_0>0$ and $c_1=(1-\gamma_S)\delta$.  
\end{CThm}
\begin{proof} For the first assertion, see Moree \cite[Theorem 4]{Mpreprint}; for the second, Serre \cite[Th\'eor\`eme 2.8]{serre}.
\end{proof}

Before stating our main results (Sec.\,\ref{Sec:StatementResults}), we recall some known facts from the literature and we explain what we mean by a ``Landau vs.~Ramanujan approximation" comparison (Sec.\,\ref{subsec:LvR}). Our focus is on the special case where $S$ is as in \eqref{Sdef}, the 
general case being discussed in greater detail by Moree \cite{LversusR}.
\subsubsection{Two claims of Ramanujan}
\label{ssec:tau}  
Put\[t_j=\begin{cases}
0 & \text{if~}q\mid \tau(j),\\
1 & \text{if~}q\nmid \tau(j).
\end{cases}\]
For $q\in\{3,5,7,23,691\},$ in his typical style, Ramanujan makes the following
claim in his famous ``unpublished" manuscript,
perhaps included with his final letter to Hardy (Jan.\,12th, 1920), 
or maybe sent to Hardy in 1923 by Francis Dewsbury, Registrar at the University of Madras.
\begin{claim}
\label{claimtau}
It is easy to prove by quite
elementary methods that
\begin{equation}
\label{simplestatement}
\sum_{j=1}^n t_j=o(n). 
\end{equation}
It can be shown by
transcendental methods that
\begin{equation}
\label{ditgaatnoggoed}
\sum_{j=1}^n t_j\sim \frac{C_q \,n}{\log^{\delta_q}n}
\end{equation}
and
\begin{equation}
\label{valseanalogie}
\sum_{j=1}^n t_j=C_q\int_1^n \frac{dx}{\log^{\delta_q}x}+O\left(\frac{n}{\log^{\rho}n}
\right),
\end{equation}
where $\rho$ is any positive number.
\end{claim}
\begin{Rem}  
We slightly changed the original
notation to make it more consistent with ours.
In order to stress the dependency on $q,$ we use 
$C_q$ and $\delta_q$. Ramanujan wrote down the values of $\delta_q$ for the above 
primes $q$ (see Table \ref{tab:Ramaprime}), and he explicitly (and 
correctly) determined $C_3,C_7$ and 
$C_{23}$ (except for a
factor $1-23^{-s}$
erroneously omitted in case $q=23$), see
Sec. \ref{sec:revisit} for details.
\end{Rem}
\begin{table}[ht]
  \begin{tabular}{l|cccccccc|}
  $q$ & $3$ & $5$ & $7$ & $23$ & $691$\\
  \hline
  $\delta_q$ & $1/2$ & $1/4$  & $1/2$  & $1/2$ & $1/690$ \\
  \end{tabular}\caption{The values $\delta_q$ for the primes studied by Ramanujan}\label{tab:Ramaprime}
  \end{table} 
\begin{Rem}
Ramanujan \cite[p.~8]{BerndtOno} claims that proving the statement 
\eqref{simplestatement} is very similar to
showing that $\pi(x)=o(x),$ 
with $\pi(x),$ the number of primes up to $x,$ 
and
refers to Landau \cite{Lehrbuch}.
Thus, one may speculate, what inspired Ramanujan in claiming that  the integral in \eqref{valseanalogie} is a  better approximation 
than \eqref{ditgaatnoggoed} might have been the fact, of which he was aware, that Gauss's approximation $\text{Li}(x)=\int_2^xdt/\log t$ is a much better estimate for $\pi(x)$ than is $x/\log x.$
\end{Rem}
\begin{Rem}
For the history of the unpublished manuscript and its wanderings, 
see Rankin \cite{Rankin82}. It was finally made available to the mathematical 
community in 1999 by Berndt and 
Ono \cite{BerndtOno}, together with 
commentaries, proofs and references to the literature.
However, the material 
related to Claim \ref{claimtau} had already been discussed years earlier by Rankin \cite{Rankin76,Rankin88}.
\end{Rem}

In his first letter to Hardy (Jan.\,16th, 1913), Ramanujan \cite[p.~24]{BR} had made a claim similar to \eqref{valseanalogie}:
\begin{claim}  \label{ramaclaim}
The number of positive integers $A\le n\le x$ that are either squares or can be written as the
sum of two squares equals
\begin{equation*}
\mathcal K\int_{A}^x \frac{dt}{\sqrt{\log t}}+\theta(x),
\end{equation*}
where $\mathcal K=0.764\dots$ and $\theta(x)$ is very small compared with 
the previous integral.  $\mathcal K$ and $\theta(x)$ have been exactly found, though
complicated...
\end{claim}
In his second letter to Hardy (Feb.\,27th, 1913), answering his inquiry 
(see \cite[p.~56]{BR}), Ramanujan
wrote: ``the order of $\theta(x)$ which you asked for in your letter is
$\sqrt{x}/\sqrt{\log x}$." \par 
See the exposition of Berndt and Rankin \cite{BR} for the full text of these two letters. 

\subsubsection{Landau vs. Ramanujan}\label{subsec:LvR}
Let $S$ be the set of natural numbers that can be written as a sum of two squares.
The fact that $S$ is a multiplicative set was already known to Fermat.
Following Landau, let us denote $S(x)$ by $B(x)$ in this particular case. 
In 1908, Landau \cite{L} proved (see also \cite[pp.~641--669]{Lehrbuch}) 
that, asymptotically,
\begin{equation}
\label{Edmund}
B(x)\sim {\mathcal K}\frac{x}{\sqrt{\log x}},
\end{equation}
a result of which Ramanujan was most likely unaware  in 1913.
\par To honor the contribution of both Landau and Ramanujan, the constant
\begin{equation}
\label{LRconstant}
{\mathcal K}= \frac{1}{\sqrt{2}} \prod_{p\equiv 3 \pmod*4} (1 - p^{-2})^{-1/2}=0.7642236535892206\ldots
\end{equation}
is called the \emph{Landau-Ramanujan constant} (cf.~Finch \cite[Section 2.3]{constants}).
\par For historical reasons delineated in this section, we will call
$$c_0\frac{x}{\log^{\delta} x}\quad  \text{and}\quad
c_0\int_2^x  \frac{dt}{\log^{\delta} t}$$
the \emph{Landau} and the \emph{Ramanujan approximation} to $S(x),$ respectively. Further, if the inequality
$$\left|S(x)- c_0\frac{x}{\log^{\delta}x}\right|< 
\bigg|S(x)- c_0\int_2^x  \frac{dt}{\log^{\delta}t} \bigg|$$
holds for every $x$ sufficiently large,
we say that the Landau approximation is better than the
Ramanujan approximation (and the other way around if the reverse inequality holds).  
Partial integration gives us
$$
c_0\int_2^x\frac{dt}{\log^{\delta}t}=\frac{c_0 ~x}{\log^{\delta}x}\left(1+
\frac{\delta}{\log x}
+O\left(\frac{1}{\log ^{2} x}\right)\right),
$$
and comparison with
\eqref{initstarrie} then yields
the following corollary of 
Classical Theorem \ref{thm:multiplicativeset}. 
\begin{Cor}
\label{halfcriterion}
If $S$ is a multiplicative set satisfying \eqref{primecondition}, the associated Euler-Kronecker
constant $\gamma_S$ exists.
If $\gamma_S<1/2,$ then Ramanujan's approximation is asymptotically better than Landau's, and the other way around if $\gamma_S>1/2.$ 
\end{Cor}
\begin{Rem} If $\gamma=1/2,$ then Landau and Ramanujan give the same approximation up to the second order term. To see which one is closer to the actual value of $S(x)$, one would have to study the higher order terms.
\end{Rem}
By partial integration we see that Claim \ref{ramaclaim} implies the asymptotic \eqref{Edmund}. 
Nevertheless, one can wonder whether Ramanujan's integral expression provides a better approximation than Landau's asymptotic.
\par The first to ever consider this question seems to have been Hardy, who in his lectures on Ramanujan's work (see \cite[pp.\,9, 63]{Hardy}) writes that Ramanujan's ``integral has no advantage, as an approximation, over the simpler function $\mathcal K x/\sqrt{\log x}.$" He also says, see \cite[p.\,19]{Hardy}, ``The integral is better replaced by the simpler function...". 
However, as revealed by Shanks \cite{Shanks}, Hardy made his claims based on a flawed paper \cite{Stanley} 
of his PhD
student, Gertrude Stanley.

Going beyond the first order
asymptotic behavior, it can be
shown (see, e.g., 
Hardy \cite[p.~63]{Hardy}) that, as $x\to\infty,$ $B(x)$ has
an asymptotic
series expansion in the sense of Poincar\'e of the form
\begin{equation}
\label{serro}
B(x)={\mathcal K}\frac{x}{\sqrt{\log x}}\left(1+\frac{c_0}{\log x}+\frac{c_1}{\log ^2 x}
+ \cdots + \frac{c_{j-1}}{\log ^j x}
+O\left(\frac{1}{\log ^{j+1} x}\right)\right),
\end{equation}
where $j$ can be taken arbitrarily large and the $c_i$ are constants.  
Serre \cite{serre} proved a similar result for a larger class of so-called \textit{Frobenian multiplicative functions}. This result implies, in particular, that for the multiplicative set  $S_{\tau;q}=\{n\le x:q\nmid \tau(n)\}$  we  asymptotically have
\begin{equation}
\label{serro2}
S_{\tau;q}(x)=C_q~\frac{x}{\log^{\delta_q}x}\left(1+\frac{c_0}{\log x}+\frac{c_1}{\log ^2 x}
+ \cdots + \frac{c_{j-1}}{\log ^j x}
+O\left(\frac{1}{\log ^{j+1} x}\right)\right),
\end{equation}
where $q$ is any of the primes studied by Ramanujan, the constants $c_i$ may depend on the choice of $q$, 
and $\delta_q$  is given in Table \ref{tab:Ramaprime}.
Much earlier, Watson \cite{Watson} (who had had the unpublished manuscript in his possession for many years) showed that 
$S_{\tau;691}(x)=O(x(\log x)^{-1/690})$. Both expansions \eqref{serro} and
\eqref{serro2} fit in the framework set up in
the opening paragraphs of this article and are special cases of \eqref{starrie}.

By partial integration, Claims \ref{claimtau} and 
\ref{ramaclaim} imply that expansions of the form 
\eqref{serro} and \eqref{serro2} should hold true for $B(x),$ respectively $S_{\tau;q}(x)$.
Both claims also imply particular values for the $c_i.$ However,
already the values of $c_0$ from \eqref{serro}  and \eqref{serro2} 
turn out to be incorrect.  
\begin{CThm}
\label{foutfout}
For  $q\in \{3,5,7,23,691\}$ the asymptotic \eqref{ditgaatnoggoed} is 
correct, cf.~Rankin \cite{Rankin76,Rankin88}, but the refined estimate \eqref{valseanalogie} is false, cf.~Moree \cite{Ramadisproof}, for every $\rho>1+\delta_q$, with 
$\delta_q$ as in Table \ref{tab:Ramaprime}. Claim \ref{ramaclaim} is 
true for
$\theta(x)=O(x\log^{-3/2}x)$, but false for $\theta(x)=o(x\log^{-3/2}x),$ 
cf.~Shanks \cite{Shanks}.
\end{CThm} 

For a  more detailed and leisurely account of the historical aspects, see Berndt and Moree \cite{brucepieter} 
or Moree and Cazaran \cite{Yana}. 
The latter authors focus on
the work done on counting integers represented by
quadratic forms other than $X^2+Y^2.$
\par The reader might wonder about what happens for the primes 
not mentioned in Classical Theorem \ref{foutfout}. Here it is known, thanks to
deep work of Serre \cite{serre74,serre}, that an asymptotic of the form 
\eqref{ditgaatnoggoed} holds. However, the correctness of the refined estimate \eqref{valseanalogie}  is an 
open problem.

\subsubsection{Ramanujan-type congruences and divisor sums}
\label{ssec:divisors} 
We now know that $691$ and the other primes studied by Ramanujan are only a few out of a 
larger, but finite set of exceptional primes modulo  which certain congruences hold for the six cusp forms
given in Table \ref{tab:formslist}.
Following the notation used by Ramanujan and, later, by Swinnerton-Dyer, we  let $Q$ and $R$ denote the normalized Eisenstein series $E_4$ and $E_6,$ which, along with $\Delta,$ are given by
$$Q=E_4=1+240\sum_{n\ge 1}\sigma_3(n)q_1^n,\quad R=E_6=1-504\sum_{n\ge 1}\sigma_5(n)q_1^n,\quad\Delta=\frac{1}{1728}
(E_4^3-E_6^2).$$
\begin{table}[ht]
\begin{tabular}{|c|cccccc|}
\hline
Weight $w$ & $12 $ & $ 16 $ & $ 18  $ & $ 20 $ & $ 22 $ & $ 26  $\\
Form & $  \Delta  $ & $ Q\Delta $ & $ R\Delta  $ & $ Q^2\Delta $ & $ QR\Delta $ & $ Q^2R\Delta  $\\
\hline
\end{tabular}\caption{The cusp forms in the ``Serre and Swinnerton-Dyer" classification}
\label{tab:formslist}
\end{table}
It is an impressive feat that Ramanujan actually found all exceptional primes for $\Delta$.
\par The weights $w$ appearing here\footnote{The traditional
notation $k$ unfortunately clashes with the subscript used in 
$\sigma_k.$} are precisely those for
which the associated space of cusp forms of the full modular group is 1-dimensional.  
\par It is well-known that the coefficients $\tau_w(n)$ of the cusp forms above satisfy the following fundamental properties 
(for references see the three excellent articles highlighting different aspects of the tau  function \cite{MurtyRR,Rankin88,S-D} in 
the proceedings of the 1987 ``Ramanujan Revisited" conference).
\begin{CThm} 
\label{fabulousproperties}
For $w\in \{12,16,18,20,22,26\}$ the following properties hold:
\begin{enumerate}[{\rm(1)}]
\item  $\tau_w$ is multiplicative; that is, $\tau_w(mn)=\tau_w(m)\tau_w(n)$ whenever $(m,n)=1$.
\item if $p$ is prime, then $\tau_w(p^{e+1})=\tau_w(p)\tau_w(p^{e})-p^{k-1}\tau_w(p^{e-1})$ for any $e\ge2$. 
\item $|\tau_w(p)|\le 2p^{(w-1)/2}$.
\end{enumerate} 
\end{CThm} 
For $\tau(n)$ (which equals $\tau_{12}(n),$ but we
will keep our old notation) these properties were conjectured by Ramanujan on basis of very scant 
numerical material. They were  a starting point for amazing and fundamental 
developments in the 20th and 21st centuries, see, e.g., the book by the
Murty brothers \cite{MM}, or
the expository article by Sujatha \cite{Sujatha}.
In addition, Ramanujan found many other congruences  for $\tau_w$ involving 
sums of divisor functions (see \cite{BerndtOno}). In the years that followed,
Deligne \cite{Deligne}, Haberland \cite{Haberland}, Serre \cite{serre69,serre73a} 
and Swinnerton-Dyer \cite{S-D-ladic,S-D}
classified all primes $q$ modulo which which congruences hold for $\tau_w,$ which are of one of the following types:
\begin{enumerate}[{\rm(i)}]
\item  $\tau_w(n)\equiv n^v \sigma_{w-1-2v}(n)\pmod*q$ for all $ (n,q)=1, $ and for some $ v\in\{0,1,2\}. $
\item $\tau_w(n)\equiv0\pmod*q$ whenever  
$\big(\frac{n}{q}\big) = -1$.
\item $p^{1-w}\tau^2_w(p)\equiv0,1,2$ or $ 4\pmod*q $ for all primes $ p\ne q. $
\end{enumerate} 
The complete list of the \emph{exceptional primes} $q$ for each of the forms in 
Table \ref{tab:formslist} is given in
Section \ref{Rcongruences}. 
Following convention, we speak about the ``Serre and Swinnerton-Dyer" classification. 
\par The congruences of type (i) suggest to investigate
the non-divisibility of $n^v\sigma_k(n),$ with $v$ and $k$ arbitrary natural
numbers.
Note that if $v\ge 1,$ we may take without loss of generality
$v=1$.
The associated counting functions  
$S_{k,q}(x)=\sum_{n\le x,~q\nmid \sigma_k(n)}1$ and $S'_{k,q}(x)=\sum_{n\le x,~q\nmid n\sigma_k(n)}1$ are the main functions of interest in this 
paper.  
\par The following elementary
result (the proof is immediate from the analysis in
Sec.\,\ref{caseb=0}) greatly simplifies our analysis.
\begin{Prop} 
\label{Prop:trivial}
The prime $q$ divides
$\sigma_{k}(n)$ if and 
only if it divides $\sigma_{(k,q-1)}(n).$
\end{Prop} 
\begin{Cor}
\label{trivialbutimportant}
It is enough to study the non-divisibility problem
for $\sigma_{k}(n)$ with $k$ dividing $q-1$.
\end{Cor}
\begin{Def}
\label{Def:cong}
If the prime $q$ divides
$a(n)$ whenever it divides $b(n),$ we write
$a(n)\cong b(n)\pmod*{q}.$
\end{Def}
Note that $\cong$ is an equivalence relation. In this notation,
Proposition \ref{Prop:trivial} can be reformulated as $\sigma_{k}(n)\cong \sigma_{(k,q-1)}(n)\pmod*{q}.$ 
For example, the congruence $\tau(n)\equiv \sigma_{11}(n)\pmod*{691}$
implies $\tau(n)\cong \sigma_{1}(n)\pmod*{691}$. 
For our purposes it is not the actual congruence
that is relevant, but the weaker $\cong$ notion. As we shall see, the
Serre and Swinnerton-Dyer classification takes, up to $\cong,$ a simpler form than with the classical congruence notion. 
\par Ramanujan, in the unpublished manuscript \cite[Sec.\,19]{BerndtOno} 
was likely the first to consider $S_{k,q}(x)$. He made three claims (also reproduced by 
Rankin \cite{Rankin76}). These were later proved by
Watson \cite{Watson}. One of these claims, namely that
$S_{k,q}(x)=O(x\log^{-1/(q-1)}x)$ in case $k$ is odd, is discussed by Hardy in his Ramanujan lectures \cite[\S 10.6]{Hardy}.
The asymptotic behavior of 
$S_{k,q}(x)$ for general $k$ was determined by Rankin \cite{Rankin61}. Eira Scourfield \cite{Scourfield64} (in her 1963 PhD thesis, 
supervised by
Rankin)
generalized his work by establishing asymptotics in the case 
where a prescribed prime power is required to exactly divide
 $\sigma_k(n)$. 
In a later paper \cite{Scourfield70} she considered the divisibility of
the divisor function by arbitrary fixed integers. 
\par In this paper we will determine the second order behavior of 
$S_{k,q}(x)$.  
In particular, one of our main results, Theorem \ref{main1}, 
gives a formula for the Euler-Kronecker constant $\gamma_{k,q}$ 
associated to the non-divisibility of $\sigma_k(n)$ by an odd prime $q,$ which allows one to decide on 
the ``Landau vs.~Ramanujan problem" for prescribed $k$ and $q$.
In case $(k,q-1)=1,$ which holds indeed for most of the exceptional primes of type (i) and the accompanying values $k$, we can invoke Theorem \ref{main2} in order to decide on the 
``Landau vs. Ramanujan problem."  
\subsubsection{Other functions}
\label{ssec:other}
Ford, Luca and Moree \cite{FordLucaMoree} 
studied the divisibility of the function $f=\varphi$ by $q,$ with $ \varphi(n) $ the Euler totient 
function, for any odd prime $q.$
They showed that Ramanujan wins for $q\le 67,$ and Landau 
for $q>67,$ and were the first to resolve this type of comparison problem
for infinitely many cases.
Earlier, Spearman and Williams \cite{SW} had determined
the relevant leading constant $C_q$ by relating it 
to the arithmetic of the cyclotomic number field 
$\mathbb Q(\zeta_q)$.

More generally, 
Scourfield 
\cite{Scourfield72}
considered integer-valued multiplicative functions $f(n)$ with the property that 
$f(p)=W(p)$ for primes $p,$ with $W(x)$ being a polynomial with integral coefficients. 
For this class of functions, she obtained
an asymptotic expression for 
$\#\{n\le x:N\nmid f(n)\}$, 
while
Narkiewicz \cite{Narkiewicz} obtained asymptotics 
for $\#\{n\le x:(f(n),N) = 1\}$.
\subsection{Statement of results}\label{Sec:StatementResults}  
Before stating our results, let us fix some terminology used in the sequel.
\begin{Def}
\label{gammaK-def}
Given a divisor $m$ of $q-1$, let 
$K_m$ be the unique subfield of $\mathbb Q(\zeta_q)$ of degree $(q-1)/m$. 
By $\mathcal{O}_{K_m},\zeta_{K_m}(s)$
and
$\gamma_{K_m},$ we denote its associated 
ring of integers, Dedekind zeta function and Euler-Kronecker constant, respectively.
\end{Def}
The uniqueness of $K_m$ is a consequence of Galois theory 
and $\textrm{Gal}(\mathbb Q(\zeta_q)/\mathbb Q)\cong (\mathbb Z/q\mathbb Z)^*$ 
being 
cyclic, see Section \ref{sec:cyc-subfields}. 
The constant $\gamma_{K_m}$ is obtained 
on setting $L_S(s)=\zeta_{K_m}(s)$ and $\alpha=1$ in \eqref{EKf}, cf.~Section \ref{Dedebasic}.
\par After this rather long introduction, we are now able to state 
our findings. The first half of Theorem \ref{main1} is a special case of 
Classical Theorem \ref{thm:multiplicativeset}, the formulas
for $\gamma_{k,q}$ and $\gamma'_{k,q}$ being the novel feature.
\begin{Thm}
\label{main1}
Let $q$ be an odd prime and $k\ge 1$ an integer. 
We define
\begin{equation*}
\label{Skn}
S_{k,q}(x)=\sum_{n\le x,~q\nmid \sigma_k(n)}1
\quad
\textrm{and}
\quad
S^\prime_{k,q}(x)=\sum_{n\le x,~q\nmid n\sigma_k(n)}1.
\end{equation*}
Put 
$r=(k,q-1)$ and assume that $h=(q-1)/r$ is even.
The counting function $S_{k,q}(x)$  satisfies an asymptotic
expansion \eqref{serro2} in the sense of Poincar\'e with $\delta_q=1/h$.
In particular, there is a positive constant
$C_{k,q}$ only depending on $r$ and $q$ such that 
\begin{equation*}
\label{finalckq}
S_{k,q}(x)=\frac{C_{k,q} ~x}{\log^{1/h} x}\left(1+
\frac{1-\gamma_{k,q}}{h \log x}
+O_{k,q}\left(\frac{1}{\log ^{2} x}\right)\right).
\end{equation*}
Here $\gamma_{k,q}$ is the Euler-Kronecker constant of the sum of divisors
function $\sigma_k(n)$ and satisfies
\begin{equation}
\label{gammakq}
\gamma_{k,q}
=
\gamma-
\frac{1}{h}
(2\gamma_{K_{2r}} - \gamma_{K_{r}}
)
-\frac{\log q}{h(q-1)}-S(r,q),
\end{equation}
where $\gamma$ is the Euler-Mascheroni constant, $\gamma_{K_{r}}$ is as in Definition \ref{gammaK-def},
\begin{align}\label{Srq-def}
S(r,q)=&-\sum_{g_p=1}\frac{(q-1)\log p}{p^{q-1}-1}+
\sum_{g_p=1}\frac{q\log p}{p^{q}-1}-\sum_{g_p\ge 3}\frac{(g_p-1)\log p}{p^{g_p-1}-1}
+\sum_{g_p\ge 3}\frac{g_p\log p}{p^{g_p}-1}\nonumber\\
&+\sum_{g_p=2}\frac{\log p}{p^2-1}+\sum_{\substack{g_p\ge4 \\ 2\mid g_p}}\frac{\log p}{p^{g_p/2}-p^{-g_p/2}},
\end{align}
$g_p$ is the multiplicative order of $p^r$ modulo $q,$ 
and
the sums are over primes $p\ne q$.
\par The counting function $S^\prime_{k,q}(x)$  also satisfies an asymptotic
expansion \eqref{serro2} in the sense of Poincar\'e with $\delta_q=1/h$. 
In particular, we have 
\begin{equation*}
\label{final}
S^\prime_{k,q}(x)=\frac{C'_{k,q} ~x}{\log^{1/h} x}\bigg(1+
\frac{1-\gamma'_{k,q}}{h\log x}
+O_{k,q}\bigg(\frac{1}{\log ^{2} x}\bigg)\bigg),
\end{equation*}
with 
\begin{equation}
\label{gamma-gammaprime-rel}
C'_{k,q}=\left(1-\frac{1}{q}\right)C_{k,q}\quad\text{and}\quad\gamma'_{k,q}=\gamma_{k,q}+ \frac{\log q}{q-1}.
\end{equation}
\end{Thm}
\begin{Rem}The remaining cases where $q=2$ or
$h$ is odd are rather trivial, see Secs. \ref{boringcases:2}--\ref{boringcases:h}.
\end{Rem}
\begin{Rem}
Consistent with Proposition \ref{Prop:trivial} we have
$C_{k,q}=C_{r,q}, C'_{k,q}=C'_{r,q},\gamma_{k,q}=\gamma_{r,q}$ and
$\gamma'_{k,q}=\gamma'_{r,q}.$ The constant $C_{k,q}$ was first 
determined by 
Rankin \cite[p.\,38]{Rankin61}.
For completeness we derive his formula again, in our notation, in 
Sec.\,\ref{sec:Ckq}.
\end{Rem}
The following is a special case of Corollary \ref{halfcriterion}.
\begin{Cor}
\label{whoisbetter}
A Ramanujan-type claim for $S_{k,q}(x)$
is false if $\gamma_{k,q}\ne 0.$
If $\gamma_{k,q}<1/2$, then the Ramanujan integral approximation 
for $S_{k,q}(x)$ is asymptotically better than the Landau asymptotic. If $\gamma_{k,q}>1/2$, then it is the other way around. The same applies for $S'_{k,q}(x)$ and its Euler-Kronecker constant $\gamma'_{k,q}$. 
\end{Cor} 
The proof of Theorem \ref{main1} rests on studying the associated
Dirichlet series $T(s):=L_S(s)$ (defined in 
\eqref{LS}) with $S=\{n\ge 1:q\nmid \sigma_k(n)\},$ and
expressing it in term of Dirichlet $L$-series and a function which is regular
for Re$(s)>1/2.$
An important aspect in our analysis will be played by the greatest common divisor  
$r=(k,q-1).$ 
For small values of $r,$ this is motivated by the congruences involving exceptional primes of type (i), for which we have $r\in\{1,3,5\}.$ 
We prove that, for any prescribed $r,$ the Landau approximation is better than the Ramanujan one for all large enough $q.$ 
\begin{Thm}
\label{Thm:r>1}
There exists an absolute constant $c_1$ 
such that for every positive integer $r,$ every prime $$q\ge e^{2r(\log r+\log \log (r+2)+c_1)},\text{\,
with\,\,}q\equiv 1\pmod*{2r},$$ and 
every positive integer $k$ satisfying $(k,q-1)=r,$ the Landau approximation is better than the Ramanujan approximation for both $S_{k,q}(x)$ and 
$S'_{k,q}(x).$ 
\end{Thm}
The larger the prime $q$ gets, the more the associated 
Dirichlet series $T(s)$ will resemble
the Riemann zeta function, and so the closer the associated Euler-Kronecker
constant approximates $\gamma.$  This is expressed more mathematically in the
following theorem.
\begin{Thm}
\label{thm:gammalimit}
Let $q$ be an odd prime and $k\ge 1$ an integer.
Put $r=(k,q-1).$
We have $$\gamma_{k,q}=\gamma+O\Big(\frac{\log^3 q}{q^{1/r}\log \log q}+\frac{r\log^2 q}{\sqrt{q}}\Big),$$ 
where the implied constant is absolute.
\end{Thm}
\begin{Cor}
Let $\epsilon>0.$ There exists a
constant $c_1(\epsilon)$ 
such that 
$$|\gamma_{k,q}-\gamma|<\epsilon,$$
for every positive integer $r,$ every prime $$q\ge e^{r(3\log r+
2\log \log (r+2)+c_1(\epsilon))},\text{\,
with\,\,}q\equiv 1\pmod*{2r},$$ 
and every positive integer $k$ satisfying $(k,q-1)=r.$
\end{Cor}
Thus for a random
choice of $k$ and $q,$ the constant $\gamma_{k,q}$ 
will be close to $\gamma,$ and as
$\gamma>1/2$ we deduce that Landau generically wins over 
Ramanujan.  
In the special case when $r=1,$ we were able to find the precise value of $q$ beyond which Landau always wins. As this value was not too large, by extensive numerical checks we were also able to determine, for each of the remaining values of $q,$ whether it is the Landau or the Ramanujan approximation that wins.
\begin{Thm}
\label{main2}
Let $k\ge 1$ be an integer and $q$ an odd prime such that $(k,q-1)=1$. 
The Landau approximation for 
$S_{k,q}(x)$ is better than the Ramanujan one for all primes $q$ other than
$$q\in \{3,5,7,11,13,17,23,29,37,41,43,47,\\53,59,73\},$$
in which cases the Ramanujan approximation is better.
The Landau approximation for  $S^\prime_{k,q}(x)$ 
is better than the Ramanujan one  
for all primes $q\ne 5.$
\end{Thm}
 
It is an exercise 
in elementary analytic number theory to show
that the number of pairs $(k,q)$ with $k,q\le x$ such
that $(k,q-1)=1$ is asymptotically equal to
$Ax^2/\log x,$ where $A=0.37399558\ldots$ is the Artin constant.
Thus, in some sense, the probability of the condition
$(k,q-1)=1$ being met, for random integers $k$ and primes $q,$ equals $A.$
\par For $r\ge2,$ our upper bound  for 
the values of $q$ beyond which the Landau approximation is certainly better increase rather dramatically (see Sec.~\ref{sec:Proof_r>1}).
Despite the considerable computer resources we had at our disposal, we were not able\footnote{For \textit{every} prescribed  $r=(q-1,k)$ it is theoretically possible to decide on the ``Landau vs.~Ramanujan approximation" for \textit{all} primes $q.$ However, we expect this would require extensive numerical checks, and, very likely, considerable improvements on the algorithms  used in this paper.} to run a test on all the remaining primes $q,$ in order to fully answer the question in case $r\in\{3,5\}.$ However, our numerical experiments strongly suggest the following.    
\begin{Con}\label{Conj_r=3} If $r=3,$ the Landau approximation for $S_{k,q}(x)$ is better than the Ramanujan one for all primes $q$ other than
$$q\in \{7,13,19,31,37,61,67,79,97,103,109,127,181\},$$ in which cases the Ramanujan approximation is better. The Landau approximation for $S'_{k,q}(x)$ is better than the Ramanujan one for all primes $q$ other than $q\in\{7,13,19,31,61,67,97,109\}.$ 
\end{Con}
\begin{Con}\label{Conj_r=5} If $r=5,$ the Landau approximation for $S_{k,q}(x)$ is better than the Ramanujan one for all primes $q$ other than
$$q\in \{11,31,41,71,101,131,241,271,311\},$$ in which cases the Ramanujan approximation is better. The Landau approximation for $S'_{k,q}(x)$ is better than the Ramanujan one for all primes $q$ other than $q\in\{11,31,71,131,241,311\}.$ 
\end{Con}

While we were not able to decide on the ``Landau vs.~Ramanujan" comparison for all primes $q$ in case $r\in\{3,5\},$ we were nevertheless able, on performing rather involved numerical checks, to answer this question for every exceptional prime $q$ and each of the six cusp forms that we studied.  
\begin{Thm}
\label{LvRmain}
Let $f=\sum_{n\ge 1}\tau_w(n)q_1^n$ be any of the six cusp forms in Table $\ref{tab:formslist}$ and let $q$ 
be any odd  exceptional prime of type {\rm(i)} or {\rm(ii)}.
If we put \[t_n=\begin{cases}
0 & \text{if~}q\mid \tau_w(n),\\
1 & \text{if~}q\nmid \tau_w(n),
\end{cases}\]
then \eqref{ditgaatnoggoed} holds for 
some positive numbers $C_q$ and $\delta_q$. However, the Ramanujan-type claim 
\eqref{valseanalogie} is false for any 
$\rho>1+\delta_q,$ 
where 
$\delta_q=r/(q-1)$
for primes $q$ of type \textup{(i)} 
(with $r$ given as in Tables 
\ref{tab:LvRnew2}--\ref{tab:LvRnew3}) and
$\delta_q=1/2$
for those of type \textup{(ii)}. 
\par Ramanujan's approximation is better than
Landau's if one of the following is satisfied: 
\begin{itemize}
\item $q=5;$
\item $q=7$ and $f\in \{\Delta,Q^2\Delta,Q^2R\Delta\};$
\item $q>5$ and $f=R\Delta.$
\end{itemize}
In all remaining cases, Landau's approximation is better. \end{Thm}

The case where $q=2$ and exceptional  is far more trivial, see Section \ref{tauboringcases} for further details.
Thanks to the work of Swinnerton-Dyer and Haberland, see
Section \ref{Haber}, we know that (iii) only occurs for
$w=16$ and $q=59$. Here we  leave computing the associated Euler-Kronecker
constant as a challenge to the interested reader. For some remarks on what happens for
non-exceptional primes, see Section \ref{non-exceptional}.
\subsection{Outline} Section \ref{sigmak}  
contains some prerequisites
 on the multiplicative order, character theory, factorization of Dedekind zeta functions and splitting of primes in certain
number fields (most of these results are well-known facts from algebraic and analytic number theory). In Section \ref{sec:ProofThm2} we evaluate the generating series and the Euler-Kronecker constant associated to $q\nmid n^v\sigma_k(n)$ and we give the proof of Theorem \ref{main1}.  In 
Section \ref{Rcongruences} we discuss the 
congruences for exceptional primes, and we 
prove Theorem \ref{LvRmain}.  
In Section \ref{sec:revisit} we look in close detail
at the Claim \ref{claimtau} statements in
the unpublished manuscript. We present our take on
why Ramanujan only wrote down $C_3,C_7$ and $C_{23}$
explicitly and give a uniform way of deriving 
his three formulae.
Section \ref{Sec:boundS(m,q)} is dedicated to finding upper and lower bounds for the sum $S(m,q).$ 
In Sections \ref{sec:Proof_r>1},  
\ref{sec:gammaapproxrate}
and \ref{sec:proofmain2} we 
prove Theorems \ref{Thm:r>1}, 
\ref{thm:gammalimit} and \ref{main2}. Section \ref{numerical-computations} discusses various aspects of the numerical computations that we carried out.
Finally, Section \ref{sec:Outlook} discusses possible
generalizations of our work and some open questions.
\par The programs used to obtain the numerical results included in this paper are
available under \url{www.math.unipd.it/~languasc/CLM.html}.

\section{Analytic and algebraic preamble}\label{sigmak} 
\subsection{Multiplicative orders} 
\label{multi} Let us recall that the letters $p$ and $q$ will be used throughout to denote prime numbers. Additionally, we assume $q$ is odd.  
For a prime $p\ne q,$ relevant for our work will be the 
\textit{multiplicative order}  
of $p$ modulo $q,$ which is the smallest
positive integer $f_p$ such that $p^{f_p}\equiv 1\pmod*{q}.$ (The order  is more commonly denoted by $\ord_q(p); $  we use $f_p$ for reasons of space  and to be consistent with the notation 
in earlier works, e.g., \cite{FordLucaMoree}.)
Obviously, if  satisfies the divisibility property $f_p\mid q-1.$
Since the order is not
defined for $p=q,$ whenever $f_p$ appears in the sequel, the implicit assumption is that $p\ne q.$  
\par Given a positive integer $m$, we let 
$g_p$ be the smallest positive integer such
that $p^{g_pm}\equiv 1\pmod*{q}$. In other words, $g_p$ is the order of $p^m$ modulo $q.$ Since this implies that
$f_p\mid g_p m$, dividing both sides by $(f_p,m),$ the greatest common divisor of
$f_p$ and $m,$ yields
$g_p=f_p/(f_p,m)$. We trivially have
\begin{equation}
    \label{trivial}
    a^m\equiv -1\pmod*q  \Leftrightarrow
 a^{2m}\equiv 1\pmod*q ~\text{and}~ a^{m}\not\equiv 1\pmod*q ,
\end{equation}
and we further note that 
\[g_p=1 \Leftrightarrow f_p\mid m \quad\text{and}\quad 
g_p=2  \Leftrightarrow f_p\mid 2m~\text{and}~f_p\nmid m.\]
Observe that if $q\equiv 1\pmod*{m},$ then $g_p$ is a divisor of
$h.$
\par We will make several times use of the following
elementary result.
\begin{Lem}
\label{lem:sol}
Let $q$ be an odd prime,  let $m$ be a divisor of $q-1$ and put $h=(q-1)/m$.
Then the equation $x^m\equiv 1\pmod*{q}$ has $m$ solutions modulo $q$. The equation $x^m\equiv -1\pmod*{q}$ has
$m$ solutions if $h$ is even, and no solutions otherwise. 
If $m$ is 
even, both congruences have at most $m/2$ prime solutions 
$p<q-2$.
\end{Lem}
\begin{proof}
Left as an exercise. Use the trivial 
fact that not both $p$ and $q-p$ can
be prime (unless $p=2$).
\end{proof}
\begin{Lem} 
\label{involvedid}
Let $m\ge 1$ be fixed.
\begin{enumerate}[{\rm 1)}]
\item We have $p^{bm}\equiv -1\pmod*{q}$ if and only if
$g_p$ is even and $b\equiv g_p/2 \pmod*{g_p}$.
\item If $g_p$ is even and
has an odd divisor $d>1,$ then  
\begin{equation}
\label{quotient2}
\frac{p^{g_p m/2}+1}{p^{g_p m/(2d)}+1} = q m_p, 
\end{equation}
with $ m_p \ge 1$ an integer.
\end{enumerate}
\end{Lem}
\begin{proof}
1) We have to find all $b$ such that 
$p^{2bm}\equiv 1\pmod*{q}$ and 
$p^{bm}\not\equiv 1\pmod*{q}$. This is equivalent
with $f_p\mid 2bm$ and $f_p\nmid bm$. On dividing
both $f_p$ and $m$ by $(f_p,m)$, we see that 
these two requirements are equivalent with
$g_p\mid 2b$ and $g_p\nmid b$. The latter two conditions are fulfilled precisely when $g_p$ is even and $b$ is an odd multiple of $g_p/2$.\\
2) The left-hand side of 
\eqref{quotient2} is easily seen to be an integer.
By part 1), the numerator is divisible by $q,$ while the denominator is not.  
Assuming otherwise, the order of $p^m$ mod
$q$ would divide $g_p/d,$ which contradicts the definition of $g_p$.
\end{proof}
 
\subsection{Cyclotomic subfields}
\label{sec:cyc-subfields}
In what follows, we fix an odd prime $q$ 
and we denote $\mathbb Q(\zeta_q)$ 
by $K$. 
By basic algebraic number theory, we have 
Gal$(\mathbb Q(\zeta_q)/\mathbb Q)\cong
(\mathbb{Z}/ q\mathbb Z)^{*}$, the latter being a cyclic group.
For every divisor $m$ of $q-1,$ there is a unique subgroup of order $m$, which, by the main theorem 
of Galois theory, corresponds uniquely to a subfield of 
$\mathbb Q(\zeta_q)$.
\begin{Def}
\label{Km-def}
For any divisor $m$ of $q-1,$
we let $K_m$ be the unique subfield of $\mathbb Q(\zeta_q)$ 
of degree $[K_m:\mathbb Q]=(q-1)/m$. Certainly, we have $K_1=K.$ 
\end{Def}

As examples, note that 
$K_2=\mathbb Q(\zeta_q+\zeta^{-1}_q)=\mathbb Q(\cos(2\pi/q))$
and $K_{q-1} = \mathbb Q$. The field
$K_2$ is the maximal real subfield
of $\mathbb Q(\zeta_q)$. Any field
$K_m$ with $m$ even is a subfield of $K_2$, and is therefore real.

By the Kronecker-Weber theorem, every abelian number field is a subfield of some cyclotomic field 
$\mathbb Q(\zeta_n)$. If we restrict $n$ to be 
a prime, we can realize precisely all extensions of 
the rationals having a cyclic Galois group that are tamely ramified in one prime and 
unramified in all other primes (note that in this case $(\mathbb Z/ n \mathbb Z)^{*}$ is always cyclic).
\par Good introductions to the arithmetic of subfields of
$\mathbb Q(\zeta_n)$ relevant to this paper can be found in the books
by Kato et al.\,\cite[Chp.\,1]{KKS} and Washington \cite[Chps.\,3-4]{Washington}.

\subsubsection{Splitting of primes}

For certain families of number fields it is not difficult to explicitly work out the
Euler product in \eqref{KEulerproduct}. For this, we need to precisely know how the rational primes
split in ${\mathcal O}_{K_m}.$ 
 
\begin{Lem}[Splitting of primes in $K_m$]
\label{Kmsplit}
Let $q$ be an arbitrary odd prime, $m$ an 
arbitrary divisor of $q-1$ and
$K_m$ the number field as  
in Definition \ref{Km-def}.
If $p\ne q$ is a prime,  the principal ideal $p \mathcal{O}_{K_m}$ 
factorizes as $p\mathcal{O}_{K_m}=\mathfrak{p}_1\cdots \mathfrak{p}_g,$ where
$g=(q-1)/(mf)$ and all prime ideals $\mathfrak{p_i}$ are distinct and of degree $f,$ with $f$ 
the multiplicative order of $p^m$ modulo $q$.
Furthermore, $q\mathcal{O}_{K_m}=\mathfrak{q}^{(q-1)/m}$ with $\mathfrak{q}$ a prime ideal of norm $q$.
\end{Lem}
\begin{proof}
See, e.g., Marcus \cite[pp.~76--78]{Marcus} or apply Theorem 5.7 of \cite{KKS}.
\end{proof}
In case $f=1,$ we say that $p$ \emph{splits completely} in 
$K_m$. This happens in $K_m$ if and only if 
$p^m\equiv 1\pmod*q$.

\begin{Prop} 
\label{sigmasplitting}
Let $k\ge 1$ be an integer and 
$q$ and odd prime. Put $r=(k,q-1)$.
We have $q\mid \sigma_k(p)$ if and only if $p$ splits completely in $K_{2r}$, but does not split completely in $K_r$.
\end{Prop}
\begin{proof}
Since $q\nmid \sigma_k(q)$ and $q$ is ramified, the assertion is correct for $p=q$, and so we may
assume $p\ne q$. By Fermat's little theorem, it suffices to verify the assertion for $k=r$.
Notice that  
$q\mid \sigma_r(p)$ if and only if
$p^{2r}\equiv 1\pmod*q$ and 
$p^{r}\not\equiv 1\pmod*q$. By Lemma \ref{Kmsplit} the proof is then concluded.
\end{proof}
The average behavior of an arithmetic function that is of rather bounded growth is very much influenced (and determined) by its values in the prime numbers. In light of this and Proposition \ref{sigmasplitting}, it is not surprising that the fields $K_{2r}$ and $K_r$ play an important role in our results and 
computations.

\subsubsection{Character theory} 
In the following, $q$ is an odd prime and $m$ a divisor of $q-1$.
\label{character-theory-sec}
\begin{Def}
\label{Cm-Xm-defs}
We let $C_m$ be the subgroup of $m$-th roots of unity inside $(\mathbb Z/ q\mathbb Z)^{*}$. As
a set we have
$$C_m=\{a\in (\mathbb Z/ q\mathbb Z)^{*}:a^m\equiv 1\pmod*q\}.$$
We have 
$K_m=\mathbb Q(\sum_{a\in C_m}\zeta_q^{a}).$
Associated to $C_m$ we define a character group, of Dirichlet characters modulo $q$, namely
$$X_m=\{\chi:\chi(a)=1\ \text{for every}\ a\in C_m\}.$$
\end{Def}
Under the Galois correspondence $C_m$ is the group associated
to $K_m$ and $K_m$ is the field belonging to $X_m.$ We have
$X_m\cong \text{Gal}(K_m/\mathbb Q),$ cf.\,Washington \cite[p.\,22]{Washington}.
Note that  $X_2$ is the set of
even characters and that 
$X_{2m} = \{\chi \in X_m : \chi\textrm{~is even}\}$ if $m\mid (q-1)/2$.
We have $\#C_m=m$ and $\#X_m=(q-1)/m$. 
The principal character, 
which we denote by $\chi_0$, is always in 
$X_m$. 
The quadratic character is 
unique and of order two and so is in $X_m$ 
if and only if $X_m$ has even order, that is if and only if $(q-1)/m$ is even. 
For notational convenience we put
\begin{equation}
\label{Xmstar}
X_m^*=X_m\setminus\{\chi_0\}=\{\chi\ne \chi_0:\chi(a)=1\ \text{for every}\ a\in C_m\}.
\end{equation}

A simple observation we will use is the following.
\begin{Lem} 
\label{charactersum}
{\rm (i)} If $\chi$ is a character modulo $m$, then
\[\sum_{a\in C_m}\chi(a)=\begin{cases}
m & \text{if~}\chi\in X_m,\\
0 & \text{otherwise}.
\end{cases}\]
{\rm (ii)} If $a\in (\mathbb{Z}/ q\mathbb Z)^{*}$, then
\[\sum_{\chi\in X_m}\chi(a)=\begin{cases}
(q-1)/m & \text{if~}a\in C_m,\\
0 & \text{otherwise}.
\end{cases}\]
\end{Lem}
\begin{proof}
(i) If $\chi\in X_m$, the claim follows from the definition and the fact that $C_m$ is of order $m$. If $\chi\not \in X_m$, then there exists $b\in C_m$ such
that $\chi(b)\ne 1$. Using the group structure of $C_m$ we then infer that
$$\chi(b)\sum_{a\in C_m}\chi(a)=\sum_{a\in C_m}\chi(ba)=\sum_{a\in C_m}\chi(a),$$ and
we conclude that $\sum_{a\in C_m}\chi(a)=0$.

(ii) If $a \in C_m$, the claim follows from the definition and the fact that $X_m$ is 
of order $(q-1)/m$. If $a\not \in C_m$, there exists $\chi_1\in X_m$ such
that $\chi_1(a)\ne 1$. Using the group structure of $X_m$ we then infer that
$$\chi_1(a)\sum_{\chi\in X_m}\chi(a)=\sum_{\chi\in X_m}(\chi_1\chi)(a)=\sum_{\chi\in X_m}\chi(a),$$ and
we conclude that $\sum_{\chi\in X_m}\chi(a)=0$.  
\end{proof}
We will often
use the trivial observation 
\eqref{trivial}, which
implies that if $r\mid (q-1)/2$, then
\begin{equation}
\label{charactersum2}
\sum_{i=1}^r\chi(a_i)=\sum_{b\in  C_{2r}} \chi(b) - \sum_{b\in  C_r} 
\chi(b),
\end{equation}
where the first sum is over the
$r$ solutions $0<a_i<q$ of $x^{r}\equiv -1\pmod*q $.
\subsection{The Dedekind zeta function}
\label{Dedebasic}
To any number field $K$ we can associate its 
\emph{Dedekind zeta function} 
\begin{equation*}
 \zeta_K(s)=\sum_{\mathfrak{a}} \frac{1}{N{\mathfrak{a}}^{s}},
\end{equation*}
defined for $\Re(s)>1$.
Here, $\mathfrak{a}$ runs over non-zero ideals in $\mathcal{O}_K$, the ring of integers of $K.$
It is known that $\zeta_K(s)$ can be analytically continued to ${\mathbb C} \setminus \{1\},$
and has a simple pole at $s=1$. 
Over $\mathcal{O}_K$  we have unique factorization into prime ideals, and this leads 
to the {\it Euler product identity}
\begin{equation}
\label{KEulerproduct}
\zeta_K(s)=\prod_{\mathfrak{p}}\frac{1}{1-N{\mathfrak{p}}^{-s}},
\end{equation}
valid for $\Re(s)>1$, where $\mathfrak{p}$ runs over the prime ideals in $\mathcal{O}_K.$ 
Around $s=1$ we have
\begin{equation}
\label{zetalaurent} 
\frac{\zeta'_K}{\zeta_K}(s)
=-\frac{1}{s-1}+\gamma_K
+O(|s-1|),
\end{equation}
and thus \eqref{EKf} holds for $\alpha=0$. 
On computing the Laurent expansion up to higher 
order, further constants, known as 
\textit{Stieltjes constants}, make their appearance 
(cf.~Lagarias \cite{Lagarias}). 

An alternative formula for
$\gamma_K$ (see, e.g., Hashimoto et al.~\cite{HIKW}) is
\begin{equation}
\label{hickup}
\gamma_K=\lim_{x\rightarrow \infty}\Big(\log x - \sum_{N\mathfrak{p}\le x}\frac{\log N\mathfrak{p}}{N\mathfrak{p}-1}\Big).
\end{equation}
It shows that the existence of (many) 
prime ideals in $\mathcal{O}_K$ of small norm has a decreasing effect on $\gamma_K$. Taking $K=\mathbb Q$ we obtain 
the well-known formula 
\begin{equation}
    \label{gammalog}
\gamma_{\mathbb Q}=\gamma=\lim_{x\rightarrow \infty}\Big(\log x
-\sum_{p\le x}\frac{\log p}{p-1}\Big).
\end{equation} 
\subsection{$L$-series factorizations}
\label{Factorizations of Dedekind zeta functions}
In what follows, we fix an odd  prime $q$ 
and put $K=K_1 =\mathbb Q(\zeta_q)$. 
We want to use more explicit factorizations of  Dedekind zeta functions. It is well-known
that
\begin{equation}\label{factorizationZeta}
\zeta_{K}(s)=\frac{1}{1-q^{-s}}\prod_{p\ne q}\left(\frac1{1-p^{-sf_p}}\right)^{(q-1)/f_p}=\zeta(s)\prod_{\chi\in X_1^*}L(s,\chi).
\end{equation}
The first identity in \eqref{factorizationZeta} is a consequence of the Euler product identity 
\eqref{KEulerproduct} and
the cyclotomic reciprocity law. 
For any prime $p\ne q,$  we put $g_p=f_p/(f_p,m)$
and $g'_p=f_p/(f_p,2m).$ Note that $g'_p=g_p/2$ if $g_p$ is even and $g'_p=g_p$ 
otherwise. The following factorization result should 
be classical, but, to our surprise, we failed to find it in the (many!) algebraic number theory textbooks we 
consulted. 
\begin{Prop}
\label{Lfactor}
If $q$ is an odd prime and $m$ divides $(q-1)/2$, then 
\begin{equation}
\label{Kmfactor}
\zeta_{K_m}(s)=\frac{1}{1-q^{-s}}\prod_{p\ne q}\left(\frac1{1-p^{-sg_p}}\right)^{(q-1)/(mg_p)}=
\zeta(s)\prod_{\chi\in X_m^*}L(s,\chi),
\end{equation}
 and
\begin{equation}
\label{K2mfactor}
\zeta_{K_{2m}}(s)=\frac1{1-q^{-s}}\prod_{p\ne q}\left( \frac1{1-p^{-sg'_p}}\right)^{(q-1)/(2mg'_p)}=\zeta(s)\prod_{\chi\in X_{2m}^*}L(s,\chi)=\zeta(s)\prod_{\substack{\chi\in X_m^*\\ \chi \ \textrm{even}}}L(s,\chi),
\end{equation}
where $X_m^*$ is defined in \eqref{Xmstar}.
\end{Prop}
\begin{proof}
We recall that $K_m$ is the associated field to $X_m.$
The first identities in both \eqref{Kmfactor} and \eqref{K2mfactor} are a consequence
of \eqref{KEulerproduct} and Lemma 
\ref{Kmsplit}; the second ones follow from Theorem 4.3 of Washington \cite{Washington}.
The final identity in \eqref{K2mfactor} 
follows on noting that $\chi\in X_{2m}$ if and only if $\chi\in X_{m}$ and $\chi$ is even 
(see Section \ref{character-theory-sec}).
\end{proof}
By comparing local factors in Proposition \ref{Lfactor} we immediately obtain the following corollary. 
\begin{Cor}
\label{Dedekindratio}
If $q$ is an odd prime and $m$ divides $(q-1)/2$, then 
\begin{equation}
\label{nognodig}
 \frac{\zeta_{K_{2m}}(s)^2}{\zeta_{K_{m}}(s)}=\frac1{1-q^{-s}}
 \prod_{2\mid g_p}\left(\frac{1+p^{-sg_p/2}}{1-p^{-sg_p/2}}\right)^{(q-1)/(mg_p)}
 =\zeta(s)\prod_{\chi\in X_m^*}L(s,\chi)^{\chi(-1)}.
\end{equation}
\end{Cor}
Our next result links $\gamma_{K_m}$ to the 
distribution of primes in residue classes modulo 
$q.$
\begin{Prop}\label{EKsump} 
If $m$ is a divisor of $q-1,$ then
\begin{align*}
 \gamma_{K_m} &= - \frac{\log q}{q-1}  + 
\lim_{x\to\infty} \Big( \log x - \frac{q-1}{m}\sum_{\substack{n\le x 
\\ n^m\equiv 1\pmod*q}} \frac{\Lambda(n)}{n} \Big)\\
&=- \frac{\log q}{q-1} - \frac{q-1}{m}\sum_{g_p \ge 2} \frac{\log p}{p^{g_p}-1} + 
\lim_{x\to\infty} \Big(\log x - \frac{q-1}{m}\sum_{\substack{p\le x 
\\ p^m\equiv 1\pmod*q}} \frac{\log p}{p-1} \Big),
\end{align*}
where $\Lambda(n)$ is the von Mangoldt function, whose values are $\log p$ if $n=p^j,$ with $j\ge1,$ and $0$ otherwise.
\end{Prop}
\begin{proof}[First proof]
Let $\chi$ be a non-principal character modulo $q$.
As $x\to \infty,$ we have the estimate
\begin{equation*}
\label{klassiekhoor}
-\frac{L'}{L} (1,\chi)
=\sum_{n=1} ^{\infty} \frac{\Lambda(n)}{n}\chi(n)=\sum_{n\le x}\frac{\Lambda(n)}{n}\chi(n)+o(1).
\end{equation*}
Further, we have the relation 
(see, e.g., \cite[\S 55]{Lehrbuch} or 
\cite[Corollary 3.9]{Tenen} 
\[
 \gamma=\log x -\sum_{n\le x} \frac{\Lambda(n)}{n} + o(1), \quad x\to \infty,
\]  
which also can be deduced from \eqref{gammalog}.
Moreover, logarithmic differentiation of the $L$-function factorization 
from \eqref{Kmfactor} yields
\begin{equation*}
\gamma_{K_m}=\gamma+
\sum_{\chi\in X_m^*}\frac{L'}{L} (1,\chi), 
\end{equation*}
where we use the fundamental fact due to Dirichlet (1837) that $L(1,\chi)\ne 0.$
On applying this identity and remarking 
that $$ \sum_{\substack{n\le x\\ (n,q)>1}} 
\frac{\Lambda(n)}{n} = \frac{\log q}{q-1} + o(1)$$ as $x\to \infty,$ we now obtain the asymptotic estimates
\begin{align*}
\gamma_{K_m}&=\log x-\sum_{n\le x} \frac{\Lambda(n)}{n}
-
\sum_{\chi\in X_m^*}\,\sum_{n\le x}\frac{\Lambda(n)}{n}\chi(n)+o(1)\\
& =  \log x-\frac{\log q}{q-1}-\sum_{\substack{n\le x\\(n,q)=1}} \frac{\Lambda(n)}{n}-
\sum_{n\le x}\frac{\Lambda(n)}{n}\sum_{\chi\in X_m^*}\chi(n)
+o(1)\\
& =  \log x-\frac{\log q}{q-1}-\sum_{n\le x}\frac{\Lambda(n)}{n}\sum_{\chi\in X_m}\chi(n)
+o(1),
\end{align*}
where in the last step we used that $\chi_0(n)=1$ if $(n,q)=1,$ and $\chi_0(n)=0$ otherwise.
The first assertion now follows on  using 
part (ii) of Lemma \ref{charactersum}.

The second assertion follows from the
first on noting that
\[
\sum_{\substack{n\le x\\ n^m\equiv 1\pmod*q}}\frac{\Lambda(n)}{n}= 
\sum_{\substack{p^{kg_p}\le x \\ k\ge 1}}
\frac{\log p}{p^{kg_p}}=
\sum_{\substack{p\le x \\ p^m\equiv 1\pmod*q}} \frac{\log p}{p-1} + \sum\limits_{g_p \ge 2} \frac{\log p}{p^{g_p}-1}+E(x),
\]
where
\[
 |E(x)|\le \sum_{a\ge 2}\sum_{p^a>x}\frac{\log p}{p^a} \le \sum_{a=2}^\infty \sum_{n > x^{1/a}} \frac{\log n}{n^a} \ll
\sum_{a=2}^\infty \frac{\log x}{a^2 x^{1-1/a}} \ll \frac{\log x}{\sqrt{x}}.\qedhere
\]
\end{proof}
\begin{proof}[Alternative proof of the second identity]
Apply (\ref{hickup}) with $K=K_m$ and the 
decomposition law in the field $K_m$ given in Lemma \ref{Kmsplit}. 
\end{proof} 
Note that it is a consequence of Dirichlet's prime number theorem  in arithmetic progressions that there exists at least one prime $p$ with $f_p=m$ (in fact, there are infinitely many).

\subsection{The quadratic case} 
Relevant for us will also be the particular case where 
$K_m$ is quadratic. This occurs for $m=(q-1)/2$, when
 we have 
$K_{(q-1)/2}=\mathbb Q(\sqrt{q^*})$,
where $q^*=(-1/q)q,$ a field of discriminant
$q^*.$
Writing 
$\chi_{q^*}(\cdot)$ for the Kronecker symbol $\kronecker{q^*}{\cdot},$
we have
$$\zeta_{K_{(q-1)/2}}(s) = \zeta(s) L(s,\chi_{q^*}),$$ 
from which we infer that  
$\gamma_{K_{(q-1)/2}},$ the Euler-Kronecker constant of $K_{(q-1)/2},$ satisfies
\begin{equation}
\label{singleL}
\gamma_{K_{(q-1)/2}} = \gamma + 
\frac{L'}{L} (1,\chi_{q^*}). 
\end{equation}
If $q\equiv 3 \pmod*{4}$, the field $K_{(q-1)/2}$ is imaginary and we can 
express $\gamma_{K_{(q-1)/2}}$ in terms of special 
values of the Dedekind $\eta$-function, see Ihara \cite[Section 2.2]{Ihara2006}. 
Assuming that the Generalized Riemann Hypothesis (GRH) holds, 
in the same paper, Ihara also proved that
$ \vert \gamma_{K_{(q-1)/2}}  \vert  \le (2+ o(1)) \log \log q.$
Murtada and Murty \cite{MourtadaM2013}
proved that there are infinitely many $q$  such that 
$ \vert \gamma_{K_{(q-1)/2}}\vert  \ge \log \log q + O(1)$,
and that, under GRH, such a bound can be sharpened to  
$\log \log q + \log \log \log q +O(1)$.
It is conjectured that  
for all the primes $q \le x $ 
we have
$\vert \gamma_{K_{(q-1)/2}} \vert \le \log \log  x + \log \log \log x +O(1)$.
Further investigations in support of such a conjecture were performed by 
Lamzouri \cite{Lamzouri2015}.

\section{Preliminary results and proof of Theorem \ref{main1}}\label{sec:ProofThm2} 

For a prime $ q, $ we want to compute the number of positive integers $n\le x$ for 
which $q\nmid f(n),$ with $f(n)=n^b\sigma_k(n),$  $b\ge 0$ and $k\in\mathbb N.$ The analysis will split in two cases, depending on whether $b=0$ or $b\ge 1.$ In the latter case, without loss of generalization we may take $b=1.$
As the case where $f(n)=n\sigma_k(n)$ is a trivial variation of the case $f(n)=\sigma_k(n),$ we will only consider it again in the
proof of Theorem \ref{main1} (see Section \ref{sec:proofThmmain1}). Let us therefore concentrate for now on studying $f(n)=\sigma_k(n)$. 
\subsection{The Dirichlet series $T(s)$}
As already explained in Section \ref{Sec:Intro}, we let
\[T(s)=\sum_{n=1}^{\infty}\frac{t_n}{n^s}\] be the associated Dirichlet series, where 
\[t_n=\begin{cases}
0 & \text{if~}q\mid \sigma_{k}(n),\\
1 & \text{if~}q\nmid \sigma_{k}(n).
\end{cases}\]
Since $\sigma_k(n)$ is multiplicative, so is  
$t_n$,
and this further implies that
$T(s)$ has an Euler product representation of the form 
\[T(s)=\prod_{p}\sum_{j=0}^{\infty}\frac{t_{p^j}}{p^{js}},\]
where the product runs over all primes $p$. In light of this, it is enough to study the divisibility of the function 
$\sigma_k(n) $ 
by a (fixed) odd prime $ q $ only 
in case $ n $ is a  prime power. 
\subsection{Divisibility of $\sigma_k$ by prime powers} 
\label{caseb=0} 
We want to determine when
\[\sigma_k(p^a)\equiv0\pmod*q.\] 
Since clearly
$\sigma_k(q^a)\equiv1\pmod*q$, we will assume from now on that
$p\ne q$. 
We have
\[\sigma_k(p^a)=1+p^k+p^{2k}+\cdots+p^{ak}=\frac{p^{k(a+1)}-1}{p^k-1},\]   
and  we note that the only values of $ a $ for which $ q\mid \sigma_k(p^a) $ are
$$ 
\begin{cases}  a\equiv-1\pmod*q&\text{if~}f_p\mid k,\\a\equiv-1
\pmod*{h_p}&\text{if~}f_p\nmid k,\end{cases}
$$ where $h_p=\frac{f_p}{(f_p,k)}.$
As $f_p\mid q-1$, 
we conclude that  
the only values of $ a $ for which $ q\mid \sigma_k(p^a) $ are
\begin{equation}
\label{divcrit}
\begin{cases}  a\equiv-1\pmod*q&\text{if~}f_p\mid r,\\a\equiv-1
\pmod*{g_p}&\text{if~}f_p\nmid r,\end{cases}
\end{equation}
where
$$r=(k,q-1)\quad\text{and}\quad g_p=\frac{f_p}{(f_p,r)}.$$  
Note that $g_p$ is the multiplicative order
of $p^{r}$ modulo $q.$
This information can be combined into a single congruence by writing 
\begin{equation}
\label{aformula3}
a\equiv-1\pmod*{\mu_p},
\end{equation} 
where
\begin{equation}
\label{gp-def}
\mu_p=\begin{cases}  q&\text{if~}g_p=1,\\
g_p &\text{if~}g_p>1.\end{cases}
\end{equation}
\subsection{An Euler product for $T(s)$} 
On  using 
\eqref{aformula3} we obtain
\begin{equation}
\label{EulerProductTsgeneral2}
T(s)=\frac1{1-q^{-s}}\prod_{p\ne q}\bigg(\sum_{j=0}^{\infty}p^{-js}-p^s\sum_{j=1}^{\infty}p^{-j\mu_ps}\bigg)=
\frac1{1-q^{-s}}\prod_{p\ne q}\frac{1-p^{-(\mu_p-1)s}}{(1-p^{-s})(1-p^{-\mu_ps})},
\end{equation}
and so
\begin{equation} 
\label{gpis2}
T(s)=\zeta(s)\prod_{p\ne q}\frac{1-p^{-(\mu_p-1)s}}{1-p^{-\mu_ps}}=\zeta(s)~D(s)\prod_{g_p=2}(1+p^{-s})^{-1},
\end{equation} 
where
\begin{equation}\label{cqs}
D(s)=\prod_{g_p\ne 2} \frac{1-p^{-(\mu_p-1)s}}
{1-p^{-\mu_ps}}.
\end{equation}
For notational convenience when using $f_p,g_p$ and $\mu_p$, we will always silently
assume that $p\ne q$. (Thus, for instance, the product in 
\eqref{cqs} is taken over the primes 
$p\ne q$ with $g_p\ne 2$.) The generating series for $T(s)$ was first found
by Rankin \cite[eq.\,(11)]{Rankin61}.
\par Using logarithmic differentiation we obtain from \eqref{gpis2} that
\begin{equation}
\label{firstlog} 
\frac{T'}{T}(s)
=\frac{\zeta'}{\zeta}(s)+\frac{D'}{D}(s)
+\sum_{g_p=2}\frac{\log p}{p^s}-\sum_{g_p=2}\frac{\log p}{p^s(p^s+1)}.
\end{equation}
For later use, we record that
\begin{equation}
\label{Dderi}
\frac{D'}{D}(1)
=\sum_{g_p\ne 2}\frac{(\mu_p-1)\log p}{p^{\mu_p-1}-1}
-\sum_{g_p\ne 2}\frac{\mu_p\log p}{p^{\mu_p}-1}.
\end{equation}
\subsection{Reformulation using $L$-series}
Our aim is next to relate the first sum on the right-hand side of \eqref{firstlog} to logarithmic derivatives
of Dirichlet $L$-series.
\begin{Lem}
\label{logp}
We have 
$$\sum_{p\equiv a\pmod*q}\frac{\log p}{p^s}=-\frac{1}{q-1}\sum_{\chi}\overline{\chi}(a)\frac{L'}{L}(s,\chi)
-\sum_{b\ge 2}\sum_{p^b\equiv a\pmod*q}\frac{\log p}{p^{bs}}.
$$
\end{Lem}
\begin{proof}
Observe that
$$ 
\sum_{n\equiv a\pmod*q}\frac{\Lambda(n)}{n^s}=\sum_{p\equiv a\pmod*q}\frac{\log p}{p^s}+\sum_{b\ge 2}\sum_{p^b\equiv a\pmod*q}\frac{\log p}{p^{bs}}. 
$$  
We further obtain
$$
\sum_{n\equiv a\pmod*q}\frac{\Lambda(n)}{n^s}=\frac{1}{q-1}\sum_{\chi}\overline{\chi}(a)\sum_{n\ge 1}\frac{\chi(n)\Lambda(n)}{n^s}=-\frac{1}{q-1}\sum_{\chi}\overline{\chi}(a)\frac{L'}{L}(s,\chi),
$$
and the proof is completed on combining these two identities.
\end{proof}

From now on we assume that $h=(q-1)/r$ is even.
This ensures that the equation
$x^{r}\equiv -1\pmod*q$ has precisely $r$ solutions 
$a_1,\ldots,a_r$ with $0<a_i<q$
(cf.~Section \ref{multi}). We observe that
$$\sum_i \sum_{b\ge 2}\sum_{p^b\equiv a_i\pmod*q}\frac{\log p}{p^{bs}}=
\sum_{\substack{p^{br}\equiv -1\pmod*q \\ b\ge 2}}\frac{\log p}{p^{bs}}.$$
By part 1) of Lemma \ref{involvedid}, the contribution of a fixed prime $p$ to the latter sum
equals
$$\sum_{n=1}^{\infty}\frac{\log p}{p^{(2n+1)s}}\quad\text{if $g_p=2,$\quad\quad and\quad\quad} \sum_{n=1}^{\infty}\frac{\log p}{p^{(2n-1)g_ps/2}}\quad\text{if $g_p\ge 4$ is even},$$
so that 
\begin{equation*}
\sum_{\substack{p^{br}\equiv -1\pmod*q \\ b\ge 2}}\frac{\log p}{p^{bs}}=
\sum_{g_p=2}\frac{\log p}{p^s(p^{2s}-1)}+
\sum_{\substack{g_p\ge 4\\ 2\mid g_p}}\frac{\log p}{p^{g_ps/2}-p^{-g_ps/2}}.
\end{equation*}
Using Lemma \ref{logp} and the
latter identity, we obtain
$$-\sum_{g_p=2}\frac{\log p}{p^s}=\frac{r\lambda(s)}{q-1}
+\sum_{g_p=2}\frac{\log p}{p^s(p^{2s}-1)}+
\sum_{\substack{g_p\ge 4\\ 2\mid g_p}}\frac{\log p}{p^{g_ps/2}-p^{-g_ps/2}},$$
where
\begin{equation*}
\label{labdie}
\lambda(s)=\frac{1}{r}\sum_{\chi}\frac{L'}{L}(s,\chi)\sum_{i=1}^r\overline{\chi}(a_i).
\end{equation*}
Combining  this with
\eqref{firstlog} we obtain
\begin{equation}
\label{weetniet}
\frac{T'}{T}(s)
=
\frac{\zeta'}{\zeta}(s)
-\frac{r\lambda(s)}{q-1}-v(s),
\end{equation}
where 
\begin{equation*}
v(s)=-\frac{D'}{D}(s)+
\sum_{g_p=2}\frac{\log p}{p^s(p^s+1)}+
\sum_{g_p=2}\frac{\log p}{p^s(p^{2s}-1)}
+
\sum_{\substack{g_p\ge4 \\ 2\mid g_p}}\frac{\log p}{p^{g_ps/2}-p^{-g_ps/2}}.
\end{equation*}
By \eqref{charactersum2} we have
\begin{equation*}
    \label{labdie2}
\lambda(s)=\frac{1}{r}\sum_{\chi}\frac{L'}{L}(s,\chi)\bigg(\sum_{a\in C_{2r}}\overline{\chi}(a)-\sum_{a\in C_r}\overline{\chi}(a)\bigg),
\end{equation*}
and so, by part (i) of Lemma \ref{charactersum},
$$
\lambda(s)=2\sum_{\chi\in X_{2r}}\frac{L'}{L}(s,\chi)-\sum_{\chi\in X_r}\frac{L'}{L}(s,\chi).$$
As $L(s,\chi_0)=\zeta(s)(1-q^{-s})$, we get
$$\frac{L'}{L}(s,\chi_0)=\frac{\zeta'}{\zeta}(s)+\frac{\log q}{q^s-1},$$
therefore
$$
\lambda(s)=\frac{\log q}{q^s-1}+
2\bigg(\frac{\zeta'}{\zeta}(s)
+ 
\sum_{\chi\in X_{2r}^*}
\frac{L'}{L}(s,\chi)\bigg)-\bigg(\frac{\zeta'}{\zeta}(s)
+ 
\sum_{\chi\in X_{r}^*}
\frac{L'}{L}(s,\chi)\bigg).$$
By Proposition \ref{Lfactor}, this can be
rewritten as
$$\lambda(s)=\frac{\log q}{q^s-1}+
2\frac{\zeta'_{K_{2r}}}{\zeta_{K_{2r}}}(s)-\frac{\zeta'_{K_{r}}}{\zeta_{K_{r}}}(s),$$
which, in combination with \eqref{weetniet}, yields
\begin{equation*} 
    \label{weetniet3}
\frac{T'}{T}(s)
=
\frac{\zeta'}{\zeta}(s) 
-
\frac{r}{q-1}\bigg(\frac{\log q}{q^s-1}+
2\frac{\zeta'_{K_{2r}}}{\zeta_{K_{2r}}}(s)-\frac{\zeta'_{K_{r}}}{\zeta_{K_{r}}}(s)\bigg)
-v(s). 
\end{equation*}
\subsection{The Euler-Kronecker constant $\gamma_{k,q}$}
For 
$K=\mathbb Q,~K_{r}$ and $K_{2r}$ we
replace $\zeta'_K/{\zeta_K}(s)$ by
the estimate given in \eqref{zetalaurent}. Noting that $v(1)=S(r,q)$,
we then obtain, on adding $$\Big(1-\frac{r}{q-1}\Big)\frac{1}{s-1}$$
to both sides of the resulting identity and on taking the limit as $s\rightarrow 1^+$,
that
\[
\gamma_{{}_T}=\gamma_{k,q}=\gamma-
\frac{r}{q-1}
\Big(\frac{\log q}{q-1} + 2\gamma_{K_{2r}} - \gamma_{K_{r}}\Big)
-S(r,q).
\]
We have thus established the
following lemma. 
\begin{Lem}
\label{EKformula}
Let $r=(k,q-1)$. If $(q-1)/r$ is even, 
then
\begin{equation}
\label{benmoe}
\gamma_{k,q}=
\gamma-
\frac{r}{q-1}
\Big(2\gamma_{K_{2r}} - \gamma_{K_{r}}+\frac{\log q}{q-1} \Big) 
-S(r,q),
\end{equation}
where $S(r,q)$ is defined in \eqref{Srq-def}. 
\end{Lem}
\begin{proof}[Second proof of Lemma \ref{EKformula}] 
Our starting  point is the Euler product from \eqref{gpis2},
which we want to express in terms of Dedekind zeta functions. 
We do this on using
\eqref{nognodig}, which on splitting off the term with $g_p=2$ 
rewrites as
\begin{equation*}
\label{nognodig2}
 \frac{\zeta_{K_{2r}}(s)^2}{\zeta_{K_{r}}(s)}=(1-q^{-s})^{-1}\prod_{g_p=2}(1+p^{-s})^h
 \prod_{g_p=2}(1-p^{-2s})^{-h/2}
 E(s)^h,
\end{equation*}
with
$$E(s):=\prod_{\substack{g_p\ge4 \\ 2\mid g_p}}\Bigl(\frac{1+p^{-sg_p/2}}{1-p^{-sg_p/2}}\Bigr)^{1/g_p}.$$ 
Combining this with \eqref{gpis2} yields 
\begin{equation}
\label{Tsh}
T(s)^h=(1-q^{-s})^{-1}\zeta(s)^hH(s)^{h/2}\zeta_{K_{r}}(s)\zeta_{K_{2r}}(s)^{-2},
\end{equation}
where $$H(s):=(D(s)E(s))^2\prod_{g_p=2}(1-p^{-2s})^{-1}.$$
Taking the Laurent series around $s=1,$ we obtain
\begin{equation}
\label{Laurent}
\frac{T'}{T}(s)+\Bigl(1-\frac{1}{h}\Bigr)\frac{1}{s-1}=\gamma+\frac{1}{2}\frac{H'}{H}(s)-\frac{1}{h}
(2\gamma_{K_{2r}}-\gamma_{K_{r}})-\frac{\log q}{h(q^s-1)}+O(s-1),
\end{equation}
where we used \eqref{zetalaurent} for each of the three zeta functions involved.
We obtain
\[ 
\frac{1}{2}\frac{H'}{H}(1)
=\frac{D'}{D}(1)+\frac{E'}{E}(1)-\sum_{g_p=2}\frac{\log p}{p^2-1},\]with \[\frac{E'}{E}(1)=-\sum_{\substack{g_p\ge4 \\ 2\mid g_p}}\frac{\log p}{p^{g_p/2}-p^{-g_p/2}},
\]
which, on recalling \eqref{Dderi} and 
\eqref{Srq-def}, shows that
$$ 
\frac{1}{2}\frac{H'}{H}(1)=-S(r,q).$$
We infer that the limit $s\to 1^+$ of the right-hand side in \eqref{Laurent} exists and
equals the right-hand side of \eqref{benmoe}. The
result then follows on invoking \eqref{EKf} with $\alpha=1-1/h$.
\end{proof}

\subsection{The proof of Theorem \ref{main1}}
\label{sec:proofThmmain1} 
With Lemma \ref{EKformula} at our disposal, we are ready to prove Theorem~\ref{main1}.
\begin{proof}
We consider
$S_{k,q}(x)$ first.
The idea is to apply Theorem \ref{thm:multiplicativeset} with $S=\{n:q\nmid \sigma_k(n)\}$, which
is a multiplicative set. 
By Proposition \ref{Prop:trivial} it follows that $S=\{n:q\nmid \sigma_r(n)\}$, with $r=(k,q-1)$.
The assumption
on $h$ ensures, see
Lemma \ref{lem:sol}, that the equation $x^r\equiv -1\pmod*{q}$ has $r$ solutions modulo $q$.
A prime $p$ is in $S$ if and only if $p^r\not\equiv -1\pmod*{q}$. 
It follows that $p$ is in $S$ if and only if $p=q$
or is in a union of $q-1-r$ arithmetic progressions modulo $q$. By a strong enough version
of the Prime Number Theorem in arithmetic progressions, we then see that 
\eqref{starrie} is satisfied with $\delta=1-r/(q-1)$. Since 
$L_S(s)=T(s)$, as defined in \eqref{EulerProductTsgeneral2}, we infer that $\gamma_S=\gamma_{r,q}=\gamma_{k,q}$. The proof of this
case is completed on invoking Lemma \ref{EKformula}.
\par For  $S'_{k,q}(x)$ the factor $(1-q^{-s})^{-1}$ in
the generating series is not there anymore, and so the associated
generating series $T'(s)$ satisfies
$T'(s)=(1-q^{-s}) T(s)$.
Logarithmic differentiation  then
yields
\[
\gamma'_{k,q}
=
\gamma_{{}_{T'}}=
\frac{\log q}{q-1}+\gamma_{{}_T}
=
\frac{\log q}{q-1}+\gamma_{k,q} ,
\]
completing the proof.  
\end{proof}

\subsection{The case $q=2$}
\label{boringcases:2} 
Let $k \ge 1$ be arbitrary.
We start by
noting that, since $\sigma_k(n)\equiv \sigma_1(n)\pmod*2,$ there is no dependency on $k$. 
It is not difficult to see that, in the cases $2\nmid \sigma_k(n)$ and $2\nmid n\sigma_k(n)$, the
generating series equal
$$\frac1{1-2^{-s}}\prod_{p>2}\frac1{1-p^{-2s}}=(1+2^{-s})\zeta(2s)
\quad\text{and}\quad
\prod_{p>2}\frac1{1-p^{-2s}}=(1-2^{-2s})\zeta(2s),$$
respectively. 
The functions $S_{k,2}(x)$ and $S_{k,2}'(x)$ count the number of integers
of the form $2^e(2m+1)^2\le x$ with $e,m\ge 0$, respectively the number of odd squares not exceeding $x$. 
It is then an easy exercise to show that 
\[
 S_{k,2}(x)=\Bigl(1+\frac{1}{\sqrt{2}}\Bigr)\sqrt{x}+O(\log x)
\quad\text{and}\quad
S_{k,2}'(x)=\frac12\sqrt{x}+O(1).
\]

\subsection{The case $2\nmid h$}
\label{boringcases:h} 
Let $q$ be an odd prime and $k\ge 1$ an integer. 
Put 
$r=(k,q-1)$ and  $h=(q-1)/r.$
The asymptotic behavior of $S_{k,q}(x)$ in case $h$ is odd was first determined
by Rankin \cite{Rankin61}. Since $p^{q-1}=p^{rh}\equiv (-1)^h\pmod*q$ it
follows that $p^r\not\equiv -1\pmod*q$, and so $g_p\ne 2$, and thus  
\eqref{gpis2}
simplifies to
$$T(s)=\zeta(s)D(s),$$
where $D(s)$ is defined in \eqref{cqs}.
It follows that, asymptotically, 
$$S_{k,q}(x)\sim D(1)\,x\quad\text{and}\quad S^\prime_{k,q}(x)\sim \Bigl(1-\frac{1}{q}\Bigr)D(1)\,x.$$
Further, 
$$\gamma_{k,q}=\gamma+\frac{D'}{D}(1)\quad\text{and}\quad \gamma'_{k,q}=\gamma+\frac{D'}{D}(1)+\frac{\log q}{q-1},$$ 
with $D'/D(1)$ as in \eqref{Dderi}. 
\subsection{The constants $C_{k,q}$ and $C'_{k,q}$}
\label{sec:Ckq} 
Let $k,q,r$ and $h$ 
be as in Sec.\,\ref{boringcases:h}.   
Assume that $h$ is even. Recall
the definition of $X_r^*$ in \eqref{Xmstar}.
\begin{Prop}
\label{leadingconstant}
We have
\begin{equation*}
C_{k,q}=\frac{(1-q^{-1})^{-1/h}}{\Gamma(1-
1/h)}\prod_{\chi\in X_r^*}L(1,\chi)^{-\chi(-1)/h}\mathfrak c_{r,q},\,\,C'_{k,q}=\Bigl(1-\frac{1}{q}\Big)C_{k,q},
\end{equation*}
with
$$\mathfrak c_{r,q}=\prod_{g_p=1}\frac{1-p^{-(q-1)}}{1-p^{-q}}
\prod_{g_p =2}
\frac{1}{\sqrt{1-p^{-2}}}
\prod_{g_p\ge 3}\frac{1-p^{-(g_p-1)}}{1-p^{-g_p}}
\prod_{\substack{g_p\ge4 \\ 2\mid g_p}}\Bigl(\frac{1+p^{-g_p/2}}{1-p^{-g_p/2}}\Bigr)^{1/g_p}.$$ 
\end{Prop}
\begin{proof}
From \eqref{Tsh} and \eqref{Dedekindratio}, we deduce that
$T(s)^h=\zeta(s)^{h-1}R(s),$ for some function
$R(s)$ that is regular for Re$(s)>1/2$ and can 
be explicitly written down. 
By a standard application
of the (Landau)-Selberg-Delange method,
see, e.g., Tenenbaum \cite[Chapter II.5]{Tenen}, we obtain,
$$C_{k,q}=\frac{R(1)^{1/h}}{\Gamma(1-1/h)},$$ and the proof is easily completed (the details are left to
the reader).
\end{proof}
\begin{Rem}
This agrees with Rankin \cite[eq.\,(16)]{Rankin61}.
However, the constant $C_{1,5},$ which he worked out
as an example (and called $A$), contains a typo;
for $L_4,$ in his
formula for $A,$ one should
read $4L_4.$
The constant $C'_{1,5}$ was independently
computed by Moree \cite{Ramadisproof}.
\end{Rem} 
\subsection{The Euler-Kronecker constant  $\gamma_{(q-1)/2,\,q}$} 
\label{EK-quadratic}
Let $q$ be an odd prime. 
As $p^{\frac{q-1}2}\equiv \big(\frac pq\big) \pmod*q$ and $g_p$ is the multiplicative order of $p^{\frac{q-1}{2}}$ modulo $q,$ we infer that
\[g_p=\begin{cases}
1 & \text{if~} \kronecker{p}{q} = 1,\\
2 & \text{otherwise.}
\end{cases}\]
In this case, formula \eqref{EulerProductTsgeneral2} specializes to 
\begin{equation}
\label{special}
T(s)=\frac1{1-q^{-s}}\prod_{\kronecker{p}{q} =-1}\frac1{1-p^{-2s}}\prod_{\kronecker{p}{q} =1}\frac{1-p^{-(q-1)s}}{(1-p^{-s})(1-p^{-qs})}.
\end{equation}
Put $q^*=\big(\frac{-1}{q}\big) q.$ 
Using quadratic
reciprocity in the form 
$\big(\frac{p}{q}\big)=\big(\frac{q^*}{p}\big),$ we infer that
$$T(s)^2=
\frac{\zeta(s)L(s,\chi_{q^*})}{1-q^{-s}}
\prod_{\kronecker{p}{q} =-1}
\frac{1}{1-p^{-2s}}
\prod_{\kronecker{p}{q} =1} 
\bigg(\frac{1-p^{-(q-1)s}}{1-p^{-qs}}\bigg)^2.
$$
By the (Landau)-Selberg-Delange method we obtain, noting that $\Gamma(1/2)=\sqrt{\pi},$
\begin{equation*}
C_{\frac{q-1}{2},q}=\sqrt{\frac{qL(1,\chi_{q^*})}{\pi(q-1)}}\prod_{\kronecker{p}{q} =-1}
\frac{1}{\sqrt{1-p^{-2}}}
\prod_{\kronecker{p}{q} =1} 
\frac{1-p^{-(q-1)}}{1-p^{-q}},\,\,C'_{\frac{q-1}{2},q}
=\Bigl(1-\frac{1}{q}\Bigr)C_{\frac{q-1}{2},q}.
\end{equation*}
We leave it to the interested reader to check that 
this coincides with the formulas given
in Proposition \ref{leadingconstant} on setting
$k=(q-1)/2.$
\par We recall that $K_{(q-1)/2}=\mathbb Q (\sqrt{q^*})$.
Using Theorem \ref{main1} and \eqref{singleL} we obtain
\[
\gamma_{\frac{q-1}{2},q}= \frac12 \gamma_{K_{\frac{q-1}{2}}}-\frac{\log q}{2(q-1)}-S\Bigl(\frac{q-1}{2},q\Bigr)
=\frac{\gamma}{2}
+\frac{1}{2}\frac{L'}{L}(1,\chi_{q^*}) 
-\frac{\log q}{2(q-1)}-S\Bigl(\frac{q-1}{2},q\Bigr).
\]
Since 
\[
S\Bigl(\frac{q-1}{2},q\Bigr)=-\sum_{\kronecker{p}{q}=1}\log p\,\Bigl(\frac{q-1}{p^{q-1}-1}- \frac{q}{p^{q}-1}\Bigr)
+\sum_{\kronecker{p}{q}=-1}\frac{\log p}{p^2-1},
\]
by formula \eqref{Srq-def}, with $r=(q-1)/2$, we finally
obtain
\begin{equation}
\label{gammaq-1half}
\gamma_{\frac{q-1}{2},q}=\frac{\gamma}{2}
+\frac{1}{2}\frac{L'}{L}(1,\chi_{q^*}) 
-\frac{\log q}{2(q-1)}
-\sum_{\kronecker{p}{q}=-1}\frac{\log p}{p^2-1}
+\sum_{\kronecker{p}{q}=1}\log p\,\Bigl(\frac{q-1}{p^{q-1}-1}- \frac{q}{p^{q}-1}\Bigr).
\end{equation}
By Proposition \ref{useful-for-computation2} we
have $S((q-1)/2,q)>0.$
In Section \ref{numerical-computations} we will describe how to efficiently compute
$\gamma_{\frac{q-1}{2},q}$ with high accuracy.

\subsubsection{Cusp form applications} 
Let $q\in \{3,7\}$. We consider the non-divisibility of $\tau$ by $q$. Using
\eqref{gammaq-1half} it can be verified that the formulas for the corresponding 
Euler-Kronecker constants $-B_t$, as given by Moree \cite{Ramadisproof}, satisfy
$$-B_t=\gamma'_{\frac{q-1}{2},q}=\gamma_{\frac{q-1}{2},q}+\frac{\log q}{q-1},$$
as expected. Another relevant case is $q=11,$ associated to the form $R\Delta$.
Finally, the cases $q=23$ and $q=31$ are relevant for the type (ii) congruences,
see Section \ref{typeii31}.

\begin{table}[ht]
\begin{tabular}{|c|l|}\hline
$\gamma$  & \phantom{---} value \\ \hline \hline
$\gamma_{1,2}$    &    $-1.370971\ldots$    \\ \hline \hline
$\gamma_{1,3}$    &    $-0.014384\ldots$     \\ \hline \hline
$\gamma_{1,5}$    &    $-0.002812\ldots$     \\ \hline \hline
$\gamma_{2,5}$    &    $\phantom{-}0.046145\ldots $    \\ \hline \hline
$\gamma_{1,7}$    &    $\phantom{-}0.388115\ldots$        \\ \hline \hline
$\gamma_{3,7}$    &    $-0.092678\ldots$\\ \hline \hline
$\gamma_{1,11}$   &    $\phantom{-}0.282623\ldots$      \\ \hline
$\gamma_{5,11}$   &    $-0.195292\ldots$ \\ \hline \hline
$\gamma_{1,13}$   &    $\phantom{-}0.400611\ldots$       \\ \hline
$\gamma_{2,13}$   &    $\phantom{-}0.581080\ldots$       \\ \hline  
$\gamma_{3,13}$   &    $-0.019200\ldots$     \\ \hline
$\gamma_{6,13}$   &    $\phantom{-}0.030107\ldots$ \\ \hline  
\end{tabular}\caption{Euler-Kronecker constants for the smallest primes}
\label{tab:LvRnew1}
\end{table}

\section{Divisibility by exceptional primes and proof of Theorem \ref{LvRmain}} 
\label{Rcongruences}
Recall that Serre and Swinnerton-Dyer  proved that the exceptional congruences are of one of the types: 
\begin{enumerate}[{\rm(i)}]
\item  $\tau_w(n)\equiv n^v\sigma_{w-1-2v}(n)\pmod*q$ for all $ (n,q)=1, $ and for some $ v\in\{0,1,2\}. $
\item $\tau_w(n)\equiv0\pmod*q$ whenever $ \big(\frac{n}{q}\big)=-1. $  
\item $p^{1-w}\tau^2_w(p)\equiv0,1,2$ or $ 4\pmod*q $ for all primes $p\ne q$. 
\end{enumerate} 
The goal of this section is to prove 
Theorem \ref{LvRmain}, our main result on the divisibility of Fourier coefficients of
cusp forms.
To this end, we invoke Theorem \ref{main1} and its
corollary for the exceptional primes satisfying condition (i). For primes of type (ii) we have the case $w=12$, $q=23,$ already worked out in 2004
by Moree \cite{Ramadisproof}, and the case $w=16$, $q=31,$ which we work out 
in Section \ref{typeii31}. Our techniques do not apply to the primes of type (iii), which satisfy a different sort of congruence criterion (see Section \ref{Haber}), and we must therefore skip their analysis.
\subsection{Congruences of type (i)} 
The exceptional primes $ q>w $ all have $v=0$ and are given in Table \ref{Table:qlargerk}.
For  $q<w$, Table 
\ref{Table:qsmallerk} gives the value of $ v $ if $ q $ is exceptional, or the word `No' if not.
These tables are taken from Swinnerton-Dyer \cite{S-D-ladic, S-D}.
\begin{table}[ht]
\begin{tabular}{|c|cccccc|}
\hline
$ w $ & $ 12 $ & $ 16 $ & $ 18  $ & $ 20 $ & $ 22 $ & $ 26  $\\
Form & $  \Delta  $ & $ Q\Delta $ & $ R\Delta  $ & $ Q^2\Delta $ & $ QR\Delta $ & $ Q^2R\Delta  $\\
$ q $ & $ 691 $ & $ 3617 $ & $ 43867  $ & $ 283,~617 $ & $ 131,~593 $ & $ 657931  $\\
\hline
\end{tabular}
\caption{Type (i): Exceptional primes with $q>w$}
\label{Table:qlargerk}
\end{table}

\begin{table}[ht] 
\begin{tabular}{|cl|rcccccccc|}
\hline
Form & $ w $ & $ q=2 $ & $ 3  $ & $ 5 $ & $ 7 $ & $ 11  $ & $ 13 $ & $ 17 $ & $ 19 $ & $ 23 $\\
$ \Delta $ & $  12  $ & $ 0 $ & $ 0 $ & $ 1 $ & $ 1 $ & No & ~ & ~ & ~& ~\\
$ Q\Delta $ & $  16  $ & $ 0 $ & $ 0  $ & $ 1 $ & $ 1 $ & $ 1 $ & No & ~ & ~& ~\\
$ R\Delta$ & $  18  $ & $ 0 $ & $ 0  $ & $ 2 $ & $ 1 $ & $ 1 $ & $1$ & No & ~& ~\\
$ Q^2\Delta $ & $  20  $ & $ 0 $ & $ 0  $ & $ 1 $ & $ 2 $ & $ 1 $ & $ 1 $ & No & No & ~\\
$ QR\Delta $ & $  22  $ & $ 0 $ & $ 0  $ & $ 2 $ & $ 1 $ & No & $ 1 $ & $ 1 $ & No & ~\\
$ Q^2R\Delta $ & $  26  $ & $ 0 $ & $ 0  $ & $ 2 $ & $ 2 $ & $ 1 $ & No & $ 1 $ & $ 1 $& No\\
\hline
\end{tabular}\caption{Type (i): Value of $v$ for the exceptional primes with $q<w$}
\label{Table:qsmallerk}
\end{table}

\subsection{The behavior of $\tau_w(q)$ for exceptional primes $q$}
The analysis of Swinnerton-Dyer only pertains to those integers $n$ coprime
to the exceptional prime $q$. We also need to understand the $q$-divisibility
of $\tau_w(q^e)$ for all 
natural numbers $e\ge 1.$ By part (2) of Classical Theorem \ref{fabulousproperties} 
we have
$\tau_w(q^e)\equiv \tau_w(q)^e \pmod*{q},$ and so either all $\tau_w(q^e)$ are
$q$-divisible, or none is.  
Using a program by Martin Raum (Julia/Nemo), but also independently, using Pari/Gp \cite{PARI2021}, we computed $\tau_w(q)$ modulo $q.$ 
\begin{Obser}\label{NumObs1}
Let $q$ be an exceptional prime for a congruence for $\tau_w$ of type (i).
If $q<w$, then
$q\mid \tau_w(q)$.
If $q>w$, then $\tau_w(q)\equiv 1\pmod*{q}$. 
\end{Obser}
Using this numerical fact, the exceptional congruences of type (i) can
be easily ``lifted" to all integers $n.$
\begin{Prop} 
\label{prop:belangrijk}
Let $q$ be exceptional of type (i) for $\tau_w$. 
If $q<w$, then 
$\tau_w(n)\equiv n^{\max\{1,v\}}\sigma_{r}(n)\pmod*q$ with 
$r=(w-1-2v,q-1)$ and $v$ as in Table \ref{Table:qsmallerk}. If $q>w$, then 
$\tau_w(n)\equiv \sigma_{r}(n)\pmod*q$ with
$r=(w-1,q-1)$.
\end{Prop}
\begin{proof}
For $v\ge 1$ the first assertion follows since, by assumption, it holds for $(n,q)=1$ and,
in addition, $q\mid \tau_w(q).$ This implies that both sides of the congruence
are divisible by $q$ if $(n,q)>1.$ Next, assume $v=0.$ By Table \ref{Table:qsmallerk} we have $q=2$ or $q=3.$
Let $r=(w-1,q-1).$ For $n$ odd we have $\tau_w(n)\equiv \sigma_r(n)
\equiv n\sigma_r(n)\pmod*2.$ As $\tau_w(2)$ is even, we 
also have $\tau_w(n)\equiv n\sigma_r(n)\pmod*2$  for even $n.$ Along the same lines, one checks that 
$\tau_w(n)\equiv n\sigma_r(n)\pmod*3$ for $n\not\equiv 2\pmod*3.$
We claim that $\tau_w(n)\equiv \sigma_r(n)
\equiv 0\equiv n\sigma_r(n)\pmod*3$ for $n\equiv 2\pmod*3.$
Such  $n$ have a prime power divisor $p^e$ with $p\equiv 2\pmod*3,$
$p^{e+1}\nmid n$ and $2\nmid e$ . Using the fact that $r$ is odd, we see that
$\sigma_r(p^e)\equiv \sum_{j=0}^e (-1)^{jr}\equiv 0\pmod*3,$ 
and hence $3\mid \sigma_r(n).$
\par In case $q>w,$ we have $v=0.$ The assertion follows on noting that the congruence
holds for $(n,q)=1,$ and that, in addition, we have
$\tau_w(q^e)\equiv \tau_w(q)^e\equiv 1\equiv\sigma_r(q^e)\pmod*{q},$
for every $e\ge 1,$ by Numerical Observation \ref{NumObs1}.
\end{proof}
Recalling Definition \ref{Def:cong}, we obtain the following corollary.
\begin{Cor} 
\label{cor:belangrijk}
Let $q$ be exceptional of type \textup{(i)} for $\tau_w$. 
If $q<w$, then 
$\tau_w(n)\cong n\sigma_{r}(n)\pmod*q$ with 
$r=(w-1-2v,q-1)$ and $v$ as in Table \ref{Table:qsmallerk}. If $q>w$, then $\tau_w(n)\cong \sigma_{r}(n)\pmod*q$ with
$r=(w-1,q-1)$.
\end{Cor} 
\begin{Rem} 
\label{Rem:6}
It is a classical result that
$\tau(n)\equiv n\sigma_1(n)\pmod*{6}.$
Since, coefficient-wise, $Q\equiv R\equiv 1\pmod*{6},$ 
we infer that 
$\tau_w(n)\equiv \tau(n)\equiv n\sigma_1(n)\pmod*{6}.$
\end{Rem}
\subsubsection{The case $q<w$} 
Corollary \ref{cor:belangrijk}  makes clear that, disregarding theoretical considerations,
working with $r$ (rather than $v$) is what matters. Doing so leads to 
Table \ref{Table:rqsmallerk}, a variant of  Table \ref{Table:qsmallerk}. In 
Table \ref{tab:LvRnew2} we give  the associated Euler-Kronecker constants with six
decimal accuracy. 

\begin{table}[ht] 
\begin{tabular}{|cl|rcccccccc|}
\hline
Form & $ w $ & $ q=2 $ & $ 3  $ & $ 5 $ & $ 7 $ & $ 11  $ & $ 13 $ & $ 17 $ & $ 19 $ & $ 23 $\\
$ \Delta $ & $  12  $ & $ 1 $ & $ 1 $ & $ 1 $ & $ 3 $ & No & ~ & ~ & ~& ~\\
$ Q\Delta $ & $  16  $ & $ 1 $ & $ 1 $ & $ 1 $ & $ 1 $ & $ 1 $ & No & ~ & ~& ~\\
$ R\Delta$ & $  18  $ & $ 1 $ & $ 1  $ & $ 1 $ & $ 3 $ & $ 5 $ & $3$ & No & ~& ~\\
$ Q^2\Delta $ & $  20  $ & $ 1 $ & $ 1  $ & $ 1 $ & $ 3 $ & $ 1 $ & $ 1 $ & No & No & ~\\
$ QR\Delta $ & $  22  $ & $ 1 $ & $ 1  $ & $ 1 $ & $ 1 $ & No & $ 1 $ & $ 1 $ & No & ~\\
$ Q^2R\Delta $ & $  26  $ & $ 1 $ & $ 1  $ & $ 1 $ & $ 3 $ & $ 1 $ & No & $ 1 $ & $ 1 $& No\\
\hline
\end{tabular}\caption{Type (i): Value of $r$ for the exceptional primes $q<w$}
\label{Table:rqsmallerk}
\end{table}

\begin{table}[ht]
\begin{tabular}{|c|c|l|}\hline
 $r$  & $q$ & \phantom{012}$\gamma'_{r,q}$\\ \hline \hline
  $1$  &  $2$ 		& $-0.677823\ldots$  \\ \hline 
 $1$  &  $3$ 		& $\phantom{-}0.534921\ldots$ \\ \hline  
 $1$  &  $5$ 		& $\phantom{-}0.399547\ldots$ \\ \hline
 $1$  &  $7$ 		& $\phantom{-}0.712434\ldots$ \\ \hline   
 $3$  &  $7$ 		& $\phantom{-}0.231640\ldots$ \\ \hline 
 $1$  &  $11$ 		& $\phantom{-}0.522413\ldots$ \\ \hline 
 $5$  &  $11$ 		& $\phantom{-}0.044497\ldots$ \\ \hline 
 $1$  &  $13$ 		& $\phantom{-}0.614357\ldots$ \\ \hline 
 $3$  &  $13$ 		& $\phantom{-}0.194544\ldots$ \\ \hline
 $1$  &  $17$ 		& $\phantom{-}0.518971\ldots$ \\ \hline
 $1$  &  $19$ 		& $\phantom{-}0.720414\ldots$ \\ \hline
\end{tabular}
\caption{Type (i): Euler-Kronecker constants for $q<w$ related to Table \ref{Table:rqsmallerk}} 
\label{tab:LvRnew2} 
\end{table}

\subsubsection{The case $q>w$} 
In this case $v=0$, $r=(w-1,q-1)$ and the relevant table is 
Table \ref{tab:LvRnew3}.

\begin{table}[ht]
\begin{tabular}{|c|c|c|c|c|l|}\hline
form & $w$ & $r$ & $q$ & $\gamma_{r,q}$ \\ \hline \hline
$\Delta $ & $12$  & $1$ & $691$ &  $0.571714\ldots$\\ \hline  
$Q\Delta $ & $16$ & $1$ & $3617$ &  $0.574566\ldots$\\ \hline    
$R\Delta $ & $18$ & $1$ & $43867$ & $0.57669\ldots.$ \\ \hline    
$Q^2\Delta $ & $20$ & $1$ & $283$ & $0.552571\ldots$ \\ \hline
$Q^2\Delta $   & $20$ & $1$ & $617$ & $0.567565\ldots$	   \\ \hline 
$QR\Delta $ & $22$ & $1$ & $131$ & $0.532695\ldots$  \\ \hline   
$QR\Delta $   & $22$ &  $1$ & $593$ & $0.568078\ldots$ \\ \hline   
$Q^2R\Delta $ & $26$ & $5$ & $657931$   &  $0.57701\ldots.$ \\ \hline
\end{tabular}\caption{Type (i): Euler-Kronecker constants for $q>w$ related to Table \ref{Table:qlargerk}} 
\label{tab:LvRnew3}
\end{table}

The computational effort in  producing this table was substantial. The computation for $\gamma_{5,657931}$ took the longest, namely about 6 days and $14$ hours
(Dell OptiPlex-3050 equipped with an Intel i5-7500 processor, 3.40GHz, 
16 GB of RAM and running Ubuntu 18.04.5) to determine the
value of $S(5,657931)$; the computation for $\gamma_{K_5}(657931)$
and $\gamma_{K_{10}}(657931)$ took less than 1 second on the same machine.
Despite this, we were not able to get more than 5
certified decimal digits.
\par The computation for $\gamma_{1,43867}$ took less time, namely about 4 days and $15$ hours;
in this case we were not able to get more than 5 certified decimal digits either.

\subsubsection{The case $f=\Delta$}  
In Table \ref{tab:LvRnew2decimals} we recomputed, with higher precision, the values
found in 2004 by Moree 
\cite{Ramadisproof} for $\Delta$     (we give
the values of Moree in our notation, which amounts to multiplying his values by minus one). 
The congruence for $q=23$ is of type (ii) and is discussed in Section \ref{typeii31}.

\begin{table}[ht]
\begin{tabular}{|c|c|c|l|l|}\hline
$q$ & type & $\gamma$  & \phantom{---} value & \phantom{-} \cite{Ramadisproof}\\ \hline \hline
$2$ & (i) & $\gamma'_{1,2}$  &  $-0.677823\ldots$ 		&  \\ \hline   
$3$ & (i) & $\gamma'_{1,3}$  &  $\phantom{-}0.534921\ldots$ 		& $0.5349\ldots$ \\ \hline   
$5$ & (i) & $\gamma'_{1,5}$  &  $\phantom{-}0.399547\ldots$ 	& $0.3995\ldots$ \\ \hline
$7$ & (i) & $\gamma'_{3,7}$  &  $\phantom{-}0.231640\ldots$&  $0.2316\ldots$   \\ \hline
$23$ & (ii) &  &  $\phantom{-}0.216691\ldots$ 		& $0.2166\ldots$ \\ \hline
$691$ & (i) & $\gamma_{1,691}$ &  $\phantom{-}0.571714\ldots$	& $0.5717\ldots$ \\ \hline
\end{tabular}\caption{Euler-Kronecker constants related to $\Delta$}
\label{tab:LvRnew2decimals}
\end{table}

\subsection{The case $q=2$} 
\label{tauboringcases}
By Proposition \ref{prop:belangrijk}, 
cf.~Remark \ref{Rem:6}, we have  $\tau_w(n)\equiv n\sigma_1(n)\pmod*2.$ Hence $\tau_w(n)$ is odd if and only if $n$ is an odd square, and so
\[\sum_{2\nmid\tau_w(n)}1=\frac{1}{2}\sqrt{x}+O(1),\]
see also Sec.\,\ref{boringcases:2}.

\subsection{Congruences of type (ii)}
\label{typeii31} 
The case $w=12$ and $q=23$ is of this type and the  
analytic number theoretical aspects of
the non-divisibility of $\tau(n)$ by 23 are
discussed by Ramanujan \cite{BerndtOno} and  Moree \cite{Ramadisproof}. There is only one further case of
this type, namely
$w=16$ and $q=31$.
The  determination of the Euler-Kronecker constant that we present here works in the same way for  $q=23$ and $q=31,$
and is based on the congruences
\begin{equation}
\label{taukpformula}
\tau_{w}(p)\equiv
\begin{cases}
\phantom{-}1 \pmod*{q}\ &\textrm{if}\ p=q;\\
\phantom{-}0 \pmod*{q}\ &\textrm{if}\ \big(\frac pq\big)=-1;\\
-1 \pmod*{q} \ &\textrm{if}\ p=2X^2+XY+ wY^2/4;\\
\phantom{-}2 \pmod*{q} \ &\textrm{if}\ p=X^2+XY+ wY^2/2,
\end{cases}
\end{equation}
where $w=(q+1)/2$,
see Swinnerton-Dyer \cite[p.\,34]{S-D-ladic}, \cite[p.\,301]{S-D} or Serre \cite{serreJordan} 
(for $q=23$). 
In 1930, a short proof using $q$-series
was given by Wilton \cite{wilton} for the exceptional prime $23,$
who also determined the values  $\tau(n)$ modulo 23 for
every positive integer $n$. According 
to Rankin \cite{Rankin76}, more modern proofs are
based on the fact that $\eta(z)\eta(23z)$ is a newform
for the group $\Gamma_0(23)$ with multiplier system 
given by the character $\chi(n)=\kronecker{n}{23}.$  
Denote by
$N_p$ the number of distinct roots modulo $p$ of
the polynomial $x^3-x-1.$ It is known that
$\tau(p)\equiv N_p-1\pmod*{23},$ cf.\,Serre \cite[p.\,437]{serreJordan} 
or \cite[pp.\,42--43]{123}.

Let ${\mathcal S}_1$ denote the set of primes $p$ 
with $\big(\frac{p}{q}\big)=-1$. 
Let ${\mathcal S}_2$ and ${\mathcal S}_3$  be the (disjoint) sets of primes
represented by the quadratic forms 
$2X^2+XY+w Y^2/4,$ respectively
$X^2+XY+wY^2/2.$
Note that the primes
$p$ in ${\mathcal S}_2\cup {\mathcal S}_3$ satisfy $\big(\frac{p}{q}\big)=1.$ 
\par By part (2) of Classical Theorem \ref{fabulousproperties} we have 
$$\tau_{w}(p^{e+1})=\tau_{w}(p)\tau_{w}(p^e)-p^{w-1}\tau_{w}(p^{e-1})
\equiv \tau_{w}(p)\tau_{w}(p^e)-\lrkronecker{p}{q}\tau_{w}(p^{e-1}) {\pmod*q},$$ with $e\ge 1.$ Using this
recurrence we can easily compute $\tau_{w}(p^e)$ modulo $q$, see Table \ref{tab:taukpe}. We then deduce that
\begin{equation}
\label{TiiT}
T_{(ii)}(s):=\sum_{q\nmid \tau_w(n)}\frac{1}{n^s}=\frac{1}{1-q^{-s}}\prod_{p\in {\mathcal S}_1}\frac{1}{1-p^{-2s}}
\prod_{p\in {\mathcal S}_2}\frac{1+p^{-s}}{1-p^{-3s}}
\prod_{p\in {\mathcal S}_3}
\frac{1-p^{-(q-1)s}}{(1-p^{-s})(1-p^{-qs})}.
\end{equation}
\begin{table}
\begin{tabular}{|c|c|r|c|c|r|c|}\hline
$e$ & 0 & 1  & 2 & 3 & 4 & 5  \\ \hline \hline
$p=q$ & 1 & 1  & 1 & 1 & 1 & 1  \\ \hline 
$p\in {\mathcal S}_1$ & 1 & 0  & 1 & 0 & 1 & 0  \\ \hline 
$p\in {\mathcal S}_2$ & 1 & $-1$  & 0 & 1 & $-1$ & 0  \\ \hline 
$p\in {\mathcal S}_3$ & 1 & 2  & 3 & 4 & 5 & 6  \\ \hline 
\end{tabular}\caption{Value of $\tau_w(p^e)$ modulo $q$}
\label{tab:taukpe}
\end{table}
By quadratic reciprocity we have
$$\prod_p \frac{1}{1-\kronecker{p}{q}p^{-s}}=\prod_p \frac{1}{1-\kronecker{-q}{p}p^{-s}}
=L(s,\chi_{-q}),$$
where $\chi_{-q}$ 
denotes the quadratic character associated with the Kronecker symbol $\kronecker{-q}{\cdot}$.
Comparison of local factors then shows that the identity
$$T_{(ii)}(s)^2= 
\frac{\zeta(s)L(s,\chi_{-q})}{1-q^{-s}}
\prod_{p\in {\mathcal S}_1}  
\frac{1}{1-p^{-2s}}
\prod_{p\in {\mathcal S}_2} 
\left(\frac{1-p^{-2s}}{1-p^{-3s}}\right)^2
\prod_{p\in {\mathcal S}_3} 
\bigg(\frac{1-p^{-(q-1)s}}{1-p^{-qs}}\bigg)^2
$$
holds true.
Computing the logarithmic derivatives of both sides and taking their limits for $s\to 1^+$, we easily deduce that
\begin{align}
\label{Tiifinal}
\gamma_{{}_{T_{(ii)}}} = 
\frac{\gamma}{2}& 
+\frac{1}{2}\frac{L'}{L}(1,\chi_{-q})
-\frac{\log q}{2(q-1)}-\sum_{\kronecker{p}{q}=-1}\frac{\log p}{p^2-1}+\sum_{p\in {\mathcal S}_2}\log p\,\left(\frac{2}{p^2-1}-\frac{3}{p^3-1}\right)\nonumber\\
&
+\sum_{p\in {\mathcal S}_3}\log p\,\left(\frac{q-1}{p^{q-1}-1}
-\frac{q}{p^q-1}\right).
\end{align}
\par We  now want to derive \eqref{Tiifinal} in a different way and, to do so,
we start by noticing that \eqref{taukpformula} can be more compactly written as 
$$
\tau_{w}(p)\equiv
\begin{cases}
\begin{aligned}
\sigma_{(q-1)/2}(p) &\pmod*{q}&\textrm{if}\ \big(\tfrac pq\big)
\ne 1 \ \textrm{or}\ p=X^2+XY+w Y^2/2, \\
-1 &\pmod*{q} &\textrm{for all other}\ p.\,\phantom{01234567901234567890}
\end{aligned}
\end{cases}
$$ 
Let $T(s)$ be the generating series associated to
the set $\{n\ge 1:q\nmid \sigma_{(q-1)/2}(n)\}.$ Note
that $\gamma_{{}_T}=\gamma_{{}_{(q-1)/2,\,q}}.$
Comparison of 
the generating series \eqref{Tii} and 
\eqref{special} shows that
\begin{equation}
\label{Tii}
T_{(ii)}(s)=T(s)\prod_{p\in {\mathcal S_2}}\frac{(1-p^{-qs})(1-p^{-2s})}{(1-p^{-(q-1)s})(1-p^{-3s})},
\end{equation}
which by logarithmic differentiation leads to
$$
\gamma_{{}_{T_{(ii)}}}=\gamma_{{}_{(q-1)/2,\,q}}+\sum_{p\in {\mathcal S_2}}\log p\,\left(\frac{2}{p^2-1}-\frac{3}{p^3-1}+\frac{q}{p^{q}-1}-\frac{q-1}{p^{q-1}-1}\right).$$
On inserting the expression \eqref{gammaq-1half} for 
$\gamma_{{}_{(q-1)/2,\,q}}$ 
in the above identity, one  obtains \eqref{Tiifinal} upon 
simplification.

\begin{Rem}
The convergence acceleration technique presented in Section \ref{specialcases} can be used 
for  the sum over the primes in ${\mathcal S}_1,$ 
but not for the prime sums over the two other sets. 
Thus, in practice, nothing truly changes for this problem.
To get six confirmed decimal digits in Table \ref{tab:LvRnew} we truncated the prime sums at $P=10^9$;
each computation required about five minutes using Pari/Gp.
\end{Rem}
 
\begin{table}
\begin{tabular}{|c|c|c|l|l|}\hline
form & $w$ & $q$  & \phantom{---} value & Moree \cite{Ramadisproof}\\ \hline \hline
$\Delta$ & $12$ & $23$  &  $0.216691\ldots$ 		& $0.2166\ldots$ \\ \hline   
$Q\Delta$ & $16$ & $31$  &  $0.156105\ldots$ 	& \\ \hline
\end{tabular}\caption{Euler-Kronecker constants for type (ii) congruences}
\label{tab:LvRnew}
\end{table}

\subsection{Congruences of type (iii)}
\label{Haber} 
Haberland \cite{Haberland}, using Galois cohomological methods, in part III 
of a series
of papers, 
proved that the case $w=16$ and $q=59$ is of
this type.  
He thus established a conjecture of
Swinnerton-Dyer who had earlier
proved that  there cannot be further cases of 
this type. Later Boylan \cite{Boylan}, and Kiming and Verrill \cite{KV} gave different proofs.
 The relevant algebraic field is
non-abelian with a non solvable Galois group, and so a factorization
of $T(s)$ as given 
in this paper, solely in 
terms of Dirichlet $L$-series and 
a regular factor, is not expected to 
exist. 
We have to leave computing the associated Euler-Kronecker constant as an open problem.

\subsection{Non-divisibility for non-exceptional primes} 
\label{non-exceptional}
The Fourier coefficient
$\tau_w(p)$ can be computed by evaluating it modulo $q$ for enough small prime $q$ and using the bound $|\tau_w(p)|\le 2p^{(w-1)/2}$.
The main result of the book \cite{EC} is that this can be done in polynomial time in $\log p$. This requires also studying congruences
for non-exceptional primes, which turns out to be way more difficult than for the 
exceptional primes and is worked out
in a 
relatively explicit way by Bosman 
\cite[Ch.~7]{EC} for some small primes. 
Put
$$g(x)=x^{12}-4x^{11}+55x^9-165x^8+264x^7-341x^6+330x^5-165x^4-55x^3+99x^2-41x-111.$$
He proves, for example 
that for $q\ne 11$ we have $11\mid \tau(q)$ if and only if the prime $q$ decomposes in the number field 
$\mathbb Q[x]/(g(x))$ as a product of primes of degree 1 and 2, with degree 2 occurring at least once. He uses these results to show
that if $\tau(n)=0,$ then $n>2\cdot 10^{19}$, making some
progress towards Lehmer's conjecture that $\tau(n)\ne 0$.

\subsection{Proof of Theorem \ref{LvRmain}} 
For the exceptional congruences of type (i) and (ii) we 
determine the associated Euler-Kronecker 
constants with
enough precision to ensure that they are non-zero. 
It follows that the corresponding variant of Ramanujan's Claim \ref{claimtau} is
false for any $r>1+\delta_q$. In each case we 
also compute them
with more than enough precision to
decide whether they are greater than $1/2$ (in which case Landau wins) or not; see Section \ref{numerical-computations} for the algorithms employed in our numerical computations, and Tables \ref{tab:LvRnew2}--\ref{tab:LvRnew2decimals} for the values.

\section{A detailed look at
the non-divisibility claims in the
unpublished manuscript}
\label{sec:revisit}
\begin{table}[ht]
\begin{tabular}{|c|c|c|c|l|c|}\hline
$q$ & $\delta_q$ & E.P.  & $C_q$ & pp. & Sec.\\ \hline \hline
$3$ & $+$ & $+$  &  $+$ 		& 22--23 & 11\\ \hline   
$5$ & $+$ & $+$  &   		&  06--08 & 2\\ \hline
$7$ & $+$ & $+$  &  $+$ 		& 11--12 & 6\\ \hline
$23$ & $+$ & $-$  &  $-$ 		&  36--37 & 17\\ \hline
$691$ & $+$ &   &   		&  24--25 & 12\\ \hline
\end{tabular}\caption{Correctness of 
non-divisibility claims from the unpublished manuscript}
\label{tab:MScorrect?}
\end{table}
Table \ref{tab:MScorrect?} lists all the non-divisibility
claims similar to Claim \ref{claimtau} made by Ramanujan in the unpublished manuscript.
They all involve the $\tau$ function (not listed are those cases where he 
only claimed bounds of the form $O(n/\log^{\delta}n)$).
The $``+"$ entry indicates a correct claim, the $``-"$ a false
one, whereas no entry indicates that no claim was made.
The first column concerns the value of $\delta_q$ 
(see Table \ref{tab:Ramaprime}),
the second the Euler product of the generating series, the third the value of the constant $C_q,$ and the two remaining ones give the pages numbers and section numbers in \cite{BerndtOno} where
the specific claims can be found. 
Rankin, using resuls from his paper
\cite{Rankin61}, confirmed the correctness
of $C_3,C_7$ and the $\delta_q$ column 
\cite[p.\,10]{Rankin76}. However, 
$C_{23}$ needs minor correction (as first 
pointed out by Moree \cite{Ramadisproof}).
The square of Ramanujan's Euler product (17.6)
for the generating series equals the right-hand side of \eqref{Tii},
but with the factor $(1-23^{-s})^{-1}$ replaced
by $(1-23^{-s})$ (it is clear from his writing that
when he writes ``all primes of the form $23a^2+b^2,$" he excludes the prime $23$).
The asymptotic constant associated to his Euler product
he calculated correctly, but it has to
be multiplied by $23/22$ in order to obtain
the true $C_{23}$. 
\par The Dirichlet series $T_q(s)$ with 
$q\in \{3,7,23\}$ are the easiest in the sense that they satisfy
$T_q(s)^2=\zeta_M(s)A(s),$ with $M$ quadratic (in fact, with
$M=\mathbb Q(\sqrt{-q})$) and $A(s)$ a regular function for
Re$(s)>1/2.$ In this case, we have $h=2$ and $\delta_q=1/2.$ 
As we have $\delta_q=q/(q^2-1)$ for non-exceptional
$q$ (see, e.g. Serre \cite[p.\,229]{serre}),
it follows that for the tau function there are no
further primes with this property. For these three primes,
$T_q(s)$ can be related to the generating series associated
to $\{n\ge 1:q\nmid \tau_{\frac{q-1}{2}}(n)\},$ and we find 
$$C_3=C'_{1,3},\,C_7=C'_{3,7},\,C_{23}=C_{11,23}\prod_{p\in {\mathcal S_2}}\frac{(1-p^{-23})(1-p^{-2})}{(1-p^{-22})(1-p^{-3})},$$
where the latter equality is immediate from \eqref{Tii}.
Using $L(1,\chi_{-3})=\pi/\sqrt{27},$  
$L(1,\chi_{-7})=\pi/\sqrt{7},$ and
$L(1,\chi_{-23})=3\pi/\sqrt{23},$
where $\chi_{-q}$ is the quadratic
character modulo $q,$ in
combination with \eqref{TiiT}, 
we get precisely the expressions
found by Ramanujan (with the caveat pointed out above for $q=23$).
\par The five Euler products 
alluded to in
Table  \ref{tab:MScorrect?} are the tip of an iceberg, Ramanujan's work being abundant with them; for an overview,
see \cite{BKW} or \cite{Ranga}. Therefore it comes
as no surprise that his unpublished manuscript also contains more Euler products than those considered here.

\section{Bounding $S(m,q)$}
\label{Sec:boundS(m,q)}
Before we begin, let us first recall that 
$m$ is a divisor of $q-1$ 
such that $h=(q-1)/m$ is even and 
\begin{equation}
  \label{useful-for-computation2}
S(m,q) =-\sum_{g_p\ne 2}\left(\frac{(\mu_p-1)\log p}{p^{\mu_p-1}-1}-
\frac{\mu_p\log p}{p^{\mu_p}-1}
\right)
+
\sum_{\substack{g_p\ge 4\\ 2\mid g_p}}\frac{\log p}{p^{g_p/2}-p^{-g_p/2}}
+\sum_{g_p=2}\frac{\log p}{p^2-1},
\end{equation}
with $g_p$ being the multiplicative order of $p^m$ modulo $q,$ and $\mu_p$ as in \eqref{gp-def}. 
Our bounds are
given in Lemmas \ref{upperboundsS(r,q)} and 
\ref{lowerboundsS(r,q)}.
They have terms with $q^{-1/m}$ in the denominator,
and thus require $m=o(\log q)$ for them to tend
to zero. Thus, one has to think of $m$ as at most
slowly growing with $q.$
Note that $g_p\mid h,$ where $h,$ for the reason
just given, will be close in size to $q-1.$ To avoid technical complications that would bring no gain, we mostly
use $g_p\le q-1$ in the sequel. 
\subsection{An upper bound for $S(m,q)$}\label{Sec:upboundS(m,q)}
In order to prove Theorem \ref{Thm:r>1} we need an upper bound for $S(m,q),$ 
which, for any fixed $m$, tends to zero as $q\to\infty.$ This is provided by  Lemma \ref{upperboundsS(r,q)}. 
\subsubsection{A trivial estimate} 
Noticing that for $j\ge 3$ we have
$$\frac{p^{j}-1}{p^{j-1}-1}>p>\frac{3}{2}\ge \frac{j}{j-1},$$
the argument of the first sum in \eqref{useful-for-computation2} is
seen to be positive. It thus suffices to find upper bounds for the second and third sum in 
\eqref{useful-for-computation2}.
We further observe that 
$$
\sum_{\substack{g_p\ge4 \\ 2\mid g_p}}\frac{\log p}{p^{g_p/2}-p^{-g_p/2}}+\sum_{g_p=2}\frac{\log p}{p^2-1}\le \sum_{\substack{p<q-2,~g_p\ge4 \\2\mid g_p}}\frac{\log p}{p^{g_p/2}-p^{-g_p/2}}
+
\sum_{\substack{p<q-2 \\ g_p=2}}\frac{\log p}{p^2-1}+\sum_{p\ge q-2}\frac{\log p}{p^2-1},$$
and denote the latter three sums by $S_1(q),$ $S_2(q),$ respectively $S_3(q).$  
\subsubsection{The sums $S_1(q)$ and $S_2(q)$}
We first give a rough estimate of the sum 
$S_0(q;\alpha)$ of the terms in
$S_1(q)$ for which $g_p\ge \alpha,$
where we will
choose $\alpha$ later (think of $\alpha$ as being of
size $O(\log q)$). 
The remainder we denote by $S_1(q;\alpha).$ 
\par In the sequel we will make use of the fact that
$\log y/(y-1)$ is decreasing for $y>1,$ and hence so
is $\log x/(x^j-1)=\log x^j/(j(x^j-1)),$ with $j\ge 1$ any 
fixed real number
and $x>1.$
We have 
$$S_0(q;\alpha):=\sum_{\substack{p<q-2,~g_p\ge \alpha \\ 2\mid g_p}}\frac{\log p}{p^{g_p/2}-p^{-g_p/2}}
\le \sum_{\substack{p<q-2 \\ \alpha\le g_p\le h}}\frac{\log p}{p^{g_p/2}-1}
\le \sum_{j=\lceil{\alpha}\rceil}^{h}\frac{jm\log 2}{2^{j/2}-1}\le 
\frac{q^2\log 2}{2(2^{\alpha/2}-1)},$$ 
where we use that there are at most $jm$ primes $p<q-2$ 
with $g_p=j.$
We split $S_1(q;\alpha)$ as
$$S_1(q;\alpha)=\sum_{\substack{p<q-2\\ g_p=2^e,~e\ge 2 \\ g_p<\alpha}} \frac{\log p}{p^{g_p/2}-p^{-g_p/2}}+
\sum_{\substack{p<q-2,~2\mid g_p \\ P(g_p)>1,~g_p<\alpha}}\frac{\log p}{p^{g_p/2}-p^{-g_p/2}}
=S_{1,1}(q;\alpha)+S_{1,2}(q;\alpha),$$ 
where $P(g_p)$ denotes the largest odd divisor of $g_p$. 
Let $j\ge 1$ be an integer. We have 
\begin{equation}
\label{gpis4j}
\sum_{\substack{p<q-2 \\ g_p=4j}}\frac{\log p}{p^{g_p/2}-p^{-g_p/2}}<\sum_{\substack{p<q-2 \\ g_p=4j}}\frac{\log p}{p^{g_p/2}-1}<
\frac{jm
\log((q-1)^{1/(2jm)})}{(q-1)^{1/m}-1}<\frac{\log q}{2((q-1)^{1/m}-1)}.
\end{equation}
To see this, we note that for any  prime
$p$ satisfying $g_p=4j$ we have $p^{2jm}\equiv-1\pmod*q,$ hence $p^{2jm}\ge q-1$, and so $p^{g_p/2}=p^{2j}\ge (q-1)^{1/m}$.
The second inequality now follows on noting that, 
by Lemma \ref{lem:sol}, there are at most $jm$ primes $p<q-2$ satisfying the congruence. 
First assume that $\nu_2(h)\ge 2$, where $\nu_2$ is the $2$-adic valuation.
If $p$ contributes
to $S_{1,1}(q;\alpha),$ then $g_p=2^e$ with 
$2\le e\le \min\{\nu_2(h),\log \alpha/\log 2\}$, and we thus infer, on 
invoking  the estimate \eqref{gpis4j}, that 
\begin{equation*}
\label{S11}
S_{1,1}(q;\alpha)
<
\frac{\alpha_1\log q }{(q-1)^{1/m}-1)},\quad\text{with}\
\alpha_1=\min\Big\{\frac{\nu_2(h)-1}{2},\frac{\log \alpha}{\log 4}\Big\}.
\end{equation*}
If $\nu_2(h)=1$, the sum $S_{1,1}(q;\alpha)$ is zero and hence the latter estimate also (trivially) holds.
\par We now turn our attention to $S_{1,2}(q;\alpha),$
and the plan is to compare 
$(p^{g_p/2}-p^{-g_p/2})^m$ with  $p^{g_pm/2}+1$, which we know to be divisible by $q$ by part 1) of Lemma \ref{involvedid}.
If $0<\beta<1$ is fixed, it is easy to see that 
\begin{equation}
\label{min}
\min_{0<x\le \beta}\frac{(1-x^2)^m}{1+x^m}
= \frac{(1-\beta^2)^m}{1+\beta^m}\ge \frac{(1-\beta^2)^m}{(1+\beta)^m}=(1-\beta)^m,
\end{equation}
which holds for any integer $m\ge 1.$ 
Since for every prime $p$ contributing to
$S_{1,2}(q)$ we have $p^{-g_p/2}\le 1/8,$
on applying \eqref{min} with $x=p^{-g_p/2}$ and $\beta=1/8$ we obtain
\begin{equation}
\label{ongelijk}
(p^{g_p/2}-p^{-g_p/2})^m\ge c^m 
(p^{g_pm/2}+1),
\end{equation}
with $c=7/8$.
Using \eqref{ongelijk} and \eqref{quotient2} with $d=P(g_p)>1$, we conclude that 
$$(p^{g_p/2}-p^{-g_p/2})^m\ge c^m 
(p^{g_pm/2}+1)\ge c^mq(p^{g_pm/(2d)}+1)\ge c^m q(p^m+1)\ge q(cp)^m.$$
Taking $m$-th roots and noting that there are at 
most $m\alpha^2/2$ primes $p$ with
$g_p<\alpha$ and $(\log x)/x$ is decreasing
for $x\ge e$, we now infer that 
\begin{equation*}
\label{boundS2,2}
S_{1,2}(q)\le \frac{8}{7q^{1/m}}\sum_{p< m\alpha^2/2}\frac{\log p}{p}<\frac{8\log (m\alpha^2/2)}{7q^{1/m}},\end{equation*}
where we used the estimate 
$\sum_{p<x}(\log p)/p<\log x,$ valid for $x>1$,
due to Rosser and Schoenfeld \cite[(3.24)]{RS}.
\par A minor variation of the argument leading to the chain of inequalities in \eqref{gpis4j} gives
\begin{equation*}\label{KeyReq}
S_2(q)=\sum_{\substack{p<q-2 \\ g_p=2}}\frac{\log p}{p^2-1}
<
\frac{m \log((q-1)^{1/m})}{(q-1)^{2/m}-1}
<
\frac{\log q}{(q-1)^{2/m}-1}.
\end{equation*}
\begin{Rem}
We used several times the fact that there are at most $jm$ primes $p<q$ for which
$g_p=j.$ In fact, there are at most $\phi(j)m$ primes with $g_p=j$.
This would lead, at the cost of mathematical complication, to only a tiny
improvement, and so we abstained from implementing it.
\end{Rem}
\subsubsection{The sum $S_3(q)$} 
The next lemma implies that, for $q\ge 7$, 
$$S_3(q)<\frac{1.053}{q-2.1},$$ 
which is rather sharp, as by the Prime Number Theorem we asymptotically have  $S_3(q)\sim q^{-1}$.
\begin{Lem} 
\label{simpleprimesum}
For $x\ge3,$ we have
\[
 \sum_{p>x}\frac{\log p}{p^2-1}< \frac{1.053}{x}.
\]
\end{Lem}
\begin{proof}
Put  $\vartheta(x)=\sum_{p\le x}\log p$ and $x_0=7481$.
For $x\ge x_0$ one has $0.98\cdot x\le \vartheta(x)\le 1.01624\cdot x$, as was
shown by Rosser and Schoenfeld \cite[Theorems 9 and 10]{RS}. 
{From} this, one easily infers that, for $x\ge x_0$,  
$$\sum_{p>x}\frac{\log p}{p^2}= \int_x^{\infty}\frac{d\vartheta(t)}{t^2} =-\frac{\vartheta(x)}{x^2}+2\int_{x}^{\infty} \frac{\vartheta(t)}{t^3}dt\le \frac{-0.98x}{x^2} +2\cdot 1.01624\int_{x}^{\infty} \frac{dt}{t^2}
\le \frac{1.0525}{x}.
$$
Since $p^2-1=p^2(1-p^{-2})\ge p^2(1-x^{-2})$ for $p>x$, for $x\ge x_0$ we obtain that 
$$\sum_{p>x}\frac{\log p}{p^2-1}< \frac{1}{1-x^{-2}}\sum_{p>x}\frac{\log p}{p^2}\le \frac{1.0525}{x(1-x_0^{-2})}\le \frac{1.053}{x}.$$ 
For $x<7481$, we explicitly calculate the sum 
using
\[
 \sum_{p > x}\frac{\log p}{p^2-1} = -\frac{\zeta'(2)}{\zeta(2)} - 
\sum_{p\le x} \frac{\log p}{p^2-1}<0.569961-\sum_{p\le x} \frac{\log p}{p^2-1}. \qedhere
\]
\end{proof}
\subsubsection{Upper estimates for $S(m,q)$}
Since there is no prime $p\equiv -1\pmod*{q}$ with
$p<q-2,$ we note that $S_2(q)=0$ in case $m=1.$
Notice that if $\alpha$ is at
most twice the smallest 
odd prime factor of $h,$ then
$S_{1,2}(q;\alpha)=0.$
On recalling that $S(m,q)=S_0(q;\alpha)+
S_{1,1}(q;\alpha)+S_{1,2}(q;\alpha)+S_2(q;\alpha)+S_3(q)$ and inserting 
the estimates for these sums derived above, we arrive at the following result for $q\ge7$ and prime;
for $q=3$ and $q=5$ we verified the upper bound 
numerically. 
\begin{Lem}\label{upperboundsS(r,q)} 
Let $q$ be an odd prime, $m$ a divisor
of $(q-1)/2$ and $h=(q-1)/m$. Then, for any $3\le \alpha\le q-1,$ 
$$S(m,q)<\frac{\alpha_1\log q }{(q-1)^{1/m}-1}+\frac{8\log (m\alpha^2)}{7q^{1/m}}+\frac{1.053}{q-2.1}+\frac{q^2\log 2}{2(2^{\alpha/2}-1)}+
\frac{\log q}{(q-1)^{2/m}-1},$$
with $\alpha_1=\min\{(\nu_2(h)-1)/2,\log \alpha/\log 4\}.$
If $\alpha$ is at most twice the smallest 
odd prime factor of $h,$ then the second term can be dropped.
The 
final term can be dropped if $m=1$. 
\end{Lem}
\begin{Cor}\label{SufficientBound}
We have $S(m,q)\ll (\log q)(\log \log q)q^{-1/m}$, where the implicit constant is absolute. 
\end{Cor}
\begin{proof}
This follows on setting $\alpha=10\log q$ (for example) and
using the trivial bound $m\le q$ in the numerator of the second 
term. 
\end{proof}
We point out that in
case $h$ satisfies $h\equiv 2\pmod*{4}$ and has only
odd prime factors exceeding $10\log q$ (for example), 
we have the sharper bound $S(m,q)\ll 1/q+(\log q)q^{-2/m}.$  
\subsection{Lower bound for $S(m,q)$}
In order to prove Theorem \ref{thm:gammalimit} we need not only the upper bound for 
$S(m,q)$ given in
Lemma \ref{upperboundsS(r,q)}, but also a lower bound. This is provided by Lemma \ref{lowerboundsS(r,q)}.
A tedious analysis gives that for $j\ge 4$ and $p\ge 2$ always
\begin{equation}
  \label{start}
-\left(\frac{(j-1)}{p^{j-1}-1}-
\frac{j}{p^{j}-1}
\right)
+\frac{1}{p^{j/2}-p^{-j/2}}\ge 0.
\end{equation}
We are thus left with finding an upper bound for
$$T(q):=\sum_{2\nmid g_p}\frac{(\mu_p-1)\log p}{p^{\mu_p-1}-1}.$$
As a digression, we make the following observation.
\begin{Prop}
\label{prop:twopower}
If $h$ is a power of two, then $S(m,q)>0.$
\end{Prop}
\begin{proof}
If $h$ is a power of two, then so is $g_p$ (which
divides $h$). It follows that $T(q)=0.$ By 
Dirichlet's theorem on primes in arithmetic progression, the final sum in the formula 
\eqref{useful-for-computation2} for $S(m,q)$ is 
strictly positive.
\end{proof}
Observe that
$$T(q)\le \sum_{\substack{g_p\ge 3,\,2\nmid g_p \\ p<q}}\frac{(g_p-1)\log p}{p^{g_p-1}-1}
+
\sum_{\substack{g_p\ge 3,\,2\nmid g_p \\ p>q}}\frac{(g_p-1)\log p}{p^{g_p-1}-1}
+
\sum_{g_p=1}\frac{(q-1)\log p}{p^{q-1}-1},$$
which we denote by $T_1(q),T_2(q)$ and $T_3(q),$ respectively.
We have
$$T_2(q)\le \sum_{p>q}\frac{2\log p}{p^2-1}+
\sum_{p>q}\frac{q\log p}{p^4-1}\ll \frac{1}{q}+\frac{1}{q^2}.$$
Reasoning as before, cf.\,the derivation of
\eqref{gpis4j}, we deduce
\begin{equation}
\label{T3q}
T_3(q)\le 
\sum_{\substack{p<q \\ p^m\equiv 1\pmod*{q}}}\frac{q\log p}{p^{q-1}-1}
+
\sum_{p>q}\frac{q\log p}{p^{q-1}-1}
\ll \frac{\log q}{q^{h-1}}+
\frac{1}{q^{q-3}}\ll \frac{\log q}{q^{h-1}}.
\end{equation}
We write $T_1(q)=T_{1,1}(q;\alpha)+T_{1,2}(q;\alpha),$ where the first sum
runs over the terms of $T_1(q)$ with $g_p<\alpha,$ where
$\alpha$ will be chosen later.
We have
\begin{equation}
\label{T12q}
T_{1,2}(q;\alpha)\le 
\sum_{\substack{p<q \\ \alpha\le g_p\le h}}\frac{g_p\log p}{p^{g_p-1}-1}
\le \sum_{j=\lceil{\alpha}\rceil}^{h}\frac{j^2m\log 2}{2^{j-1}-1}\ll 
\frac{mh^3}{2^{\alpha-1}}\ll \frac{qh^2}{2^{\alpha}}.
\end{equation}
The sum $T_{1,1}(q;\alpha),$ we write as 
$V_1(q;\alpha)+V_2(q;\alpha),$ where in the
first sum we impose the additional condition
that $g_p$ is a prime itself.
\par We always have that
$q$ divides $(p^{g_pm}-1)/(p^m-1),$ a 
quotient which is bounded above by 
$2p^{(g_p-1)m}.$ Letting 
$\sigma'(n;\alpha)$ denote
the sum of the prime divisors $p_i<\alpha$ 
of $n,$ we then
obtain in the usual way
$$V_1(q;\alpha)\ll 
\frac{\sigma'(q-1;\alpha)\log q}
{q^{1/m}}.$$
We finally turn our attention to $V_2(q;\alpha).$
The plan is to compare 
$p^m(p^{g_p-1}-1)^m$ with  $p^{g_pm}-1$, which 
is of course divisible by $q.$ For $p\ge 2$
and $g_p\ge 3,$ we find, with $c=3/4,$
$$p^m(p^{g_p-1}-1)^m\ge c^m 
(p^{g_pm}-1)\ge c^mq(p^{g_pm/d}-1)\ge c^m q(p^{3m}-1)\ge 7qc^mp^{3m}/8,$$
where we use that $g_p$ is a composite odd
integer, and so it must have a divisor
$1<d<g_p.$
Taking $m$-th roots we infer that
\begin{equation*} 
V_2(q;\alpha)\ll \frac{1}{q^{1/m}}\sum_{g_p\le \alpha}\frac{g_p\log p}{p^2}\ll \frac{\alpha}{q^{1/m}}.\end{equation*}
We trivially have 
\begin{equation}
\label{sumofprimes}
\sigma'(q-1;\alpha)\le \sum_{p\le \alpha}p\ll \frac{\alpha^2}
{\log \alpha}.
\end{equation}
Gathering all the bounds together and setting 
$\alpha=10\log q,$ we obtain that there is an absolute
constant $c_1>0$ such that
$$-S(m,q)\le c_1\frac{\sigma'(q-1;10\log q)\log q}{q^{1/m}},$$
which on invoking \eqref{sumofprimes} leads to the
following conclusion.
\begin{Lem}
\label{lowerboundsS(r,q)}
There is an absolute
constant $c_2>0$ such that
$$-S(m,q)\le \frac{c_2\log^3 q}{q^{1/m}\log \log q}.$$
\end{Lem}
\begin{Rem}
In case $h$ satisfies $h\equiv 2\pmod*{4}$ and has only
odd prime factors exceeding $C\log q$ 
we can do much better
and using the estimates \eqref{T3q} and \eqref{T12q} obtain 
$$-S(m,q)\le T_3(q)+T_{1,2}(q;C\log q)\ll \frac{\log q}{q^{h-1}}+\frac{qh^2}{q^{C\log 2}}\ll \frac{h^2}{q^{C\log 2-1}}.$$
The final estimate follows on noting that
$C\log 2<C\log q\le h/2.$
\end{Rem}
\begin{Rem}
Suppose there are infinitely many 
primes $q\equiv 3\pmod*{4}$ 
with $q-1$ squarefree and having all its odd prime divisors
in the interval $[\log q,10\log q].$ Note that for
these primes $\sigma'(q-1;10\log q)\gg \log^2 q/(\log \log q)$ 
and so the upper bound 
\eqref{sumofprimes} with $\alpha=10\log q$ is sharp.
\end{Rem}
\begin{Rem}
It is also possible to do the 
estimation without making use of 
inequality \eqref{start}.
For that, we have to bound from above the sum
$$\sum_{g_p\ge 3}\frac{(\mu_p-1)\log p}{p^{\mu_p-1}-1},$$
where now the terms with $2\mid g_p$  are
included as well. Reasoning as 
in the derivation of \eqref{gpis4j}, we find that
$$\sum_{\substack{g_p\ge 4\\ 2\mid g_p}}\frac{(g_p-1)\log p}{p^{g_p-1}-1}\ll \frac{\log^3 q}{q^{3/2}},$$
which is swamped by the major contribution to
the error term for $-S(m,q).$
\end{Rem}

\section{Proof of Theorem \ref{Thm:r>1}}\label{sec:Proof_r>1} 
The arguments are inspired by the proof of \cite[Theorem 1]{FordLucaMoree} and are related to the number of zeros of Dirichlet $L$-series in certain 
regions near the line $\Re(s) =1$. 
McCurley \cite[Theorem 1.1]{McCurl} showed that, for every $q,$ the region 
\[\Re(s)\ge1-\frac{1}{R\log \max\{q,q|\Im(s)|,10\}},\] where $R=9.645908801,$ contains at most one zero 
$\beta_0$ of $\prod_{\chi \pmod*q}L(s,\chi)$.  
If $\beta_0$ is exceptional, it must be real, simple, and 
satisfy $L(\beta_0,\chi_q)=0,$ where $\chi_q$ is the real, nonprincipal quadratic character modulo $q.$ 
We will need an explicit version of Page's theorem 
\cite{Page1935} giving a lower bound for $\beta_0$. 
For this we use the one established by
Ford et al.\,\cite{FordLucaMoree}. 
\begin{Lem}[{\cite[Lemma 3]{FordLucaMoree}}] 
\label{estimateBeta0}
If $q\ge10\,000$ is prime and $\beta_0$ an exceptional zero, then
\[\beta_0\ge1-\frac{3.125\min\{2\pi,(\log q)/2\}}{\sqrt q\log^2q}\ge 0.9983.\]
\end{Lem}
Let $q$ be a prime, $a$ an integer coprime with $q,$  
and let
 \[\psi(x;q,a)=\sum_{\substack{n\le x \\ n\equiv a\pmod*q}}\Lambda(n)\]
be the Chebyshev  
$\psi$-function. The following modification of \cite[Lemma 9]{FordLucaMoree} is  an essential 
ingredient in our arguments (the notation used is as introduced in the beginning of this section).
\begin{Lem} 
\label{LemmaFlorian}
Let $ q\ge 10\,000 $ be a prime and $a$ a 
fixed integer coprime with $q$. For $ x\ge \exp(R\log^2 q)$ we have
$$
\Bigl\vert 
\psi(x;q,a)-\frac{x}{q-1} \Bigr \vert 
\le 
\frac{1.012x^{\beta_0}}{q} 
+
\frac{8}{9} x 
\sqrt{\frac{\log x}{R}}
\exp \Bigl(-\sqrt{\frac{\log x}{R}}\Bigr),
$$
where the first term is there only if there is an
exceptional zero $\beta_0$.
\end{Lem}  
\begin{proof} For $a=1$ this is \cite[Lemma 9]{FordLucaMoree}. The proof depends heavily on
earlier work of McCurley \cite{McCurl}, whose arguments work
for arbitrary $a$ coprime to $q.$  This allows
for an easy adaptation of the proof in 
\cite{FordLucaMoree} to any $a$ coprime with $q$ 
as well.
\end{proof}
\begin{Rem}
For primes 
$q\equiv 1\pmod*{2r},$
our interest is more precisely in 
 \[\sum_{\substack{n\le x \\ n^r\equiv -1\pmod*q}}\Lambda(n).\]
One could hope that this can be expressed as a linear
combination of $\psi(x,\chi)=\sum_{n\le x} \Lambda(n) \chi(n)$ not involving the quadratic character
modulo $q,$ thus avoiding the contribution of the
possible exceptional zero $\beta_0$.  
However, 
this is not the case by the remark after Definition \ref{Cm-Xm-defs}. 
\end{Rem}
With these ingredients in place, we can finally prove Theorem \ref{Thm:r>1}.
\begin{proof}[Proof of Theorem \ref{Thm:r>1}] 
Recall that  $r=(k,q-1)$. 
The equation
$x^{r}\equiv -1\pmod*q$ has precisely $r$ solutions 
$a_1,\ldots,a_r$, with $0<a_i<q$ 
(cf.~Section \ref{multi}).
On combining Proposition \ref{EKsump} and Lemma \ref{EKformula}, we have
\begin{equation}
\label{eq:gammakqAP-streamlined}
\gamma_{k,q}=\gamma-\sum_{i=1}^r\lim_{x\to\infty}\Bigl(\frac{\log x}{q-1} - \sum_{\substack{n\le x 
\\ n\equiv a_i\pmod*q}} \frac{\Lambda(n)}{n}\Bigr)-S(r,q).
\end{equation}
Writing $E(t;q,a):=\psi(t;q,a)-t/(q-1)$, where $(a,q)=1$, and using a partial summation argument, 
we obtain 
\begin{equation*}\label{limMangoldtai-streamlined}
\lim_{x\to\infty}\Bigl(\sum_{\substack{y<n\le x\\n\equiv a\pmod*q}}\frac{\Lambda (n)}{n}-\frac{\log (x/y)}{q-1}\Bigr)=\int_{y}^{\infty}\frac{E(t;q,a)}{t^2}dt -\frac{E(y;q,a)}{y} .
\end{equation*}
Invoking Lemma \ref{LemmaFlorian}, 
we obtain,  
for any $y\ge \exp(R\log^2q)$ the estimate
\begin{equation}\label{integralMangoldtai-streamlined}
\left| \int_{y}^{\infty}\frac{E(t;q,a)}{t^2}dt-\frac{E(y;q,a)}{y} \right|\le 
\frac{1.012(2-\beta_0)y^{\beta_0-1}}{(1-\beta_0)q}+\frac89 \frac{2RW^2+(4R+1)W+4R}{e^W},
\end{equation}
with $W=\sqrt{\log y/R}$
and where the first term can be left out if there 
is no exceptional zero $\beta_0$.
On ignoring the summands from \eqref{eq:gammakqAP-streamlined} with $n\le y,$ 
we can now use 
\eqref{eq:gammakqAP-streamlined}--\eqref{integralMangoldtai-streamlined} 
and $\beta_0\ge 0.9983,$ to obtain
\begin{equation}
\label{fundamental-m-streamlined}
\gamma_{k,q}\ge \gamma-r\Bigl(\frac{\log y}{q-1}
+ 
\frac{1.015}{D\sqrt q}y^{-\frac{D}{\sqrt q\log^2q}}\log^2q
+
\frac89\frac{2RW^2+(4R+1)W+4R}{e^W} \Bigr)-S(r,q),
\end{equation}
for any $q\ge10\,000,$ where $D=3.125\min\{2\pi,\log q/2\}$ and $y=\exp(1.44 R \log^2 q)$.
The largest of the terms in between the brackets 
in \eqref{fundamental-m-streamlined} is coming from the
exceptional zero and is $O(q^{-1/2}\log^2q).$
Using Corollary \ref{SufficientBound} we 
thus conclude that there exist absolute constants $c_2$ and $c_3$ such that 
$$
\gamma_{k,q}\ge \gamma-c_2\frac{r\log^2 q}{\sqrt{q}}-c_3\frac{\log^2 q}{q^{1/r}}=\gamma-F(q),
$$
say.
It is easy to see that there exists
an absolute constant $c_1$ such that 
$F(q)<0.077$ and hence 
$\gamma_{k,q}>1/2$ for any 
$$q\ge e^{2r(\log r+\log \log (r+2)+c_1)},\quad\text{with~} q\equiv 1\pmod*{2r}.$$ 
By Corollary \ref{whoisbetter} it then follows
that the Landau approximation
is better for any such value $q.$
Using \eqref{gamma-gammaprime-rel} we have $\gamma'_{k,q}=\gamma_{k,q}
+\log q/(q-1)>\gamma_{k,q},$ and so we obtain the same 
conclusion for $\gamma_{k,q}'.$
\end{proof}   
\begin{Rem}
\label{q-bounds-remark}
Let $q_0(r)$ be the minimal prime such that $\gamma_{r,q}>1/2$ for $q \ge q_0(r)$ using \eqref{fundamental-m-streamlined}
and  Lemma \ref{upperboundsS(r,q)}.  
Choosing $C=10$, a 
numerical evaluation of such formulae gave
$q_0(1) = 28\,537$;
$q_0(2) = 1\,160\,893$; 
$q_0(3) = 2\,089\,575\,931$; 
$q_0(r) > 10^{10}$ 
for $r\ge 4$. 
These bounds are too large 
in order for $S(r,q)$ to be evaluated over
the whole range $3\le q \le q_0(r)$, $q$ prime, $q\equiv 1 \pmod*{2r}$, as described in Section \ref{numerical-computations};
in fact there we will explain that we are able to compute $S(r,q)$ only for $1\le r\le 6$ and $3\le q\le 3000$.
However, for $r=1$ we can also use some already computed data on $\gamma_{K_{1}}$ and $\gamma_{K_{2}}$ 
to prove that $\gamma_{1,q}>1/2$ for every odd prime $q\in[q_1(1), q_0(1)]$, where $q_1(1)<3000$; see the 
proof of Theorem \ref{main2}. 
Unfortunately, the cases with $r\ge 2$ are well beyond our computational capabilities and hence
we presently cannot settle the truth of Conjectures \ref{Conj_r=3}--\ref{Conj_r=5}.
\end{Rem}
 
\section{Proof of Theorem \ref{thm:gammalimit}}
\label{sec:gammaapproxrate}
Our proof will make use of the following result.
\label{sec:gammalimit}
\begin{Prop}\label{sumsmallp}
 If $y\ge 10q$ and $q\ge 11$, then
\[
\sum_{\substack{2q<p\le y \\ p\equiv a\pmod*{q}}} \frac{\log p}{p-1}
\le \frac{2\log y + 2(\log q)\log\log (y/q)}{q-1}.
\]
\end{Prop}
\begin{proof}
In \cite[Prop.\,6]{FordLucaMoree} this is proved for $a=1.$ As it hinges on the Montgomery-Vaughan sharpening of the Brun-Titchmarsh
theorem, which holds for arbitrary progressions, it trivially
generalizes.
\end{proof}
\begin{proof}[Proof of Theorem \ref{thm:gammalimit}]
The argument leading  to 
\eqref{fundamental-m-streamlined} is easily adapted to obtain,
for any $y\ge \exp(R\log^2q),$ the upper bound
\begin{equation}
\label{fundamental-m-streamlined2}
\gamma_{k,q}\le \gamma+r\Bigl(\frac{\log y}{q-1}
+ 
\frac{1.015}{D\sqrt q}y^{-\frac{D}{\sqrt q\log^2q}}\log^2q
+
\frac89\frac{2RW^2+(4R+1)W+4R}{e^W} \Bigr)-S(r,q)+T_q(y),
\end{equation}
with $$T_q(y)=\sum_{\substack{n\le y 
\\ n^r\equiv -1\pmod*q}} \frac{\Lambda(n)}{n}.$$
Note that for the lower bound we had dropped the sum $T_q(y).$
Put $y_1=\exp(1.44 R \log^2 q)$.
Using Proposition \ref{sumsmallp}, we deduce that
$$T_q(y_1)\le T_q(2q)+\sum_{\substack{2q<p\le y_1 \\ p^r\equiv -1\pmod*{q}}} \frac{\log p}{p-1}\ll \frac{\log q}{q^{1/r}}+
\frac{\log^2 q}{q}.$$
This estimate, together with \eqref{fundamental-m-streamlined} and 
\eqref{fundamental-m-streamlined2}, then yields
$$\gamma_{k,q}=\gamma-S(r,q)+O\Big(\frac{r\log^2 q}{\sqrt{q}}+\frac{\log q}{q^{1/r}}\Big).$$
Taking into account the upper and lower bound for $S(r,q)$ provided
by Lemmas \ref{upperboundsS(r,q)} and
\ref{lowerboundsS(r,q)}, the proof is completed.
\end{proof}

\section{Proof of Theorem \ref{main2}}\label{sec:proofmain2}
We  work here under the assumption that $r=1$; that is, we study the divisibility of $n^v\sigma_k(n)$ by primes $q$ such that $(k,q-1)=1$.
We will follow the same argument used in Theorem \ref{Thm:r>1} to prove that Landau wins for large enough primes $q,$ but, in addition, we will be able to treat all the remaining primes $q$ and to conclude, in each case, whether the Landau or the Ramanujan approximation is better. For this, we will need the upper estimate established in Lemma \ref{upperboundsS(r,q)} and the following 
sandwich bounds for $\gamma_{K_1}$ and $\gamma_{K_2}.$
\begin{Lem}[{\cite[Section 6]{LanguascoR2021}}]
\label{lambdaqnumerics-30000}
For $3\le q<30\,000,$  we have
\begin{align*} 
0.3145\cdot \log q & \le \gamma_{K_1}\le 1.6270 \cdot\log q,\\
0.5254\cdot \log q &\le \gamma_{K_2}\le 1.4263 \cdot\log q. 
\end{align*}
\end{Lem}
\begin{proof}
The values $\gamma_{K_1}$ and $ \gamma_{K_2}$ (denoted by $\mathfrak G_q$ and $\mathfrak G_q^+$ in \cite{LanguascoR2021}) are  the Euler-Kronecker constants of the fields $K_1=\mathbb Q(\zeta_q)$ and $K_2=\mathbb Q(\zeta_q+\zeta_q^{-1})$ respectively, see Section \ref{sec:cyc-subfields}. The lower and upper estimates given here are taken from \cite[Section 6]{LanguascoR2021}. 
\end{proof}

We are now ready to 
prove Theorem \ref{main2}.
\begin{proof}[Proof of Theorem \ref{main2}] 
Setting $r=1$ in \eqref{fundamental-m-streamlined} gives 
\begin{equation}
\label{thm4-start}
\gamma_{1,q}\ge\gamma -\frac{\log y}{q-1} 
-\frac{1.015}{D\sqrt q}y^{-\frac{D}{\sqrt q\log^2q}}\log^2q
-\frac89\frac{2RW^2+(4R+1)W+4R}{e^W}
-S(1,q),
\end{equation}
where $q\ge10\,000$ and we recall that $D=3.125\min\{2\pi,\log q/2\}$, $y= \exp(1.44 R\log^2q)$, $W=\sqrt{\log y/R}$ and $R=9.645908801$. 
A quick numerical check using Lemma \ref{upperboundsS(r,q)} and \eqref{thm4-start} 
reveals that $\gamma_{1,q}>1/2$ for $q\ge 29\,100$.
In the remaining $q$-range we use the alternative expression 
\begin{equation}
\label{gammakq_r=1-alt}
\gamma_{1,q}=
\gamma- \frac{\log q}{(q-1)^2}-\frac{2\gamma_{K_{2}}-\gamma_{K_{1}}}{q-1}-S(1,q),
\end{equation}
which comes from 
taking $r=1$ in \eqref{gammakq}. 
Inserting the upper bound for 
$S(1,q)$ given in Lemma \ref{upperboundsS(r,q)}
 in
\eqref{gammakq_r=1-alt}, 
we obtain 
a lower bound for $\gamma_{1,q},$ which, using 
Lemma \ref{lambdaqnumerics-30000} and a numerical verification, is seen  
to exceed $1/2$
for $600<q\le 29\,000$.
Thus, to prove the first part of the statement, 
it remains to check it for $3\le q \le 600,$
which we do by a 
direct numerical evaluation of the quantities appearing in
\eqref{gammakq_r=1-alt}.
Using \eqref{gamma-gammaprime-rel} we have  $\gamma'_{1,q}=\gamma_{1,q}+\log q/(q-1)>\gamma_{1,q}$
and hence, for $q>600$, the second part of Theorem \ref{main2}
follows immediately.
In the remaining $q$-range a numerical verification completes the proof.\end{proof}
 \section{On the numerical computations} 
\label{numerical-computations}
All the numerical results presented in this paper were obtained using
the following considerations. The computation of $\gamma_{k,q}$ 
naturally splits in two  parts: the evaluation of the pair 
$\gamma_{K_r}$, $\gamma_{K_{2r}}$, and that of $S(r,q)$,
where $r=(k,q-1)$ and $h=(q-1)/r$ is even (and so $r\mid (q-1)/2$).
In fact, both problems can be handled in a more general setting, 
i.e., for each $m\mid (q-1)/2$.

We first remark that a logarithmic differentiation of the $L$-function factorization 
from \eqref{Kmfactor}--\eqref{K2mfactor} yields
\begin{equation*}
\gamma_{K_m}=\gamma+
\sum_{\chi\in X_m^*}
\frac{L'}{L}(1,\chi), \quad  
\gamma_{K_{2m}}=\gamma+
\sum_{\chi\in X_{2m}^*}
\frac{L'}{L}(1,\chi).
\end{equation*}
These formulae suggest that $\gamma_{K_m}$ and $\gamma_{K_{2m}}$ can be computed by adapting the approach presented
in \cite{Languasco2021a, LanguascoR2021}.  Indeed, using 
techniques from \cite{Languasco2021a, LanguascoR2021}
we can get the values of 
$L'/L(1,\chi)$ 
for every non-principal  Dirichlet character mod $q$. So, after having obtained the list of the divisors $m$ 
of $(q-1)/2$, in order to get  $\gamma_{K_m}$ and $\gamma_{K_{2m}},$ it is enough to sum 
$L'/L(1,\chi)$ 
on every non-principal character of $X_m$ and, respectively, $X_{2m}$.
Such sets of characters can be described in the following way:
recalling that $q$ is prime, it is enough to get $g$, a primitive root of $q$,
and $\chi_1$, the Dirichlet character mod $q$ given by
$\chi_1(g) = \exp(2\pi i/(q-1))$, to see that the set of the non-principal
characters mod $q$ is $\{\chi_1^j \colon j=1,\dotsc,q-2\}$.
In order for $\chi_1^j$ to be in $X_m$, we need that  $\chi_1^j (a) = 1$ for 
every $a\in C_m$. But  $a\in C_m$ if and only if it can be written as 
$a \equiv g^b \pmod*{q}$, with $b=\ell(q-1)/m$ for some $\ell=0,\dotsc,m-1$.
Hence 
$\chi_1^j\in X_m$ implies that $\chi_1^j (a) = \exp(2\pi i j b/(q-1)) =
\exp(2\pi i j \ell/m) = 1 $
for every  $\ell=0,\dotsc,m-1$, and this is equivalent to $m \mid j$.
Summarizing, we can say that  
$X_m = \{ \chi_0\} \cup \{\chi_1^j \colon  j=1,\dotsc,q-2; \ m \mid j \}$.
This characterization, albeit elementary, is particularly useful in practice
since the condition $m\mid j$ can be easily checked by a computer program.

Recalling that $\chi \in X_{2m}$ if and only if $\chi \in X_{m}$ and $\chi$ is even, we
observe that $\gamma_{K_{2m}}$ can be obtained without further efforts 
by storing the sum over even characters used
for $\gamma_{K_m}$.

In this way it is then possible to  evaluate every $\gamma_{K_{m}}$ and
$\gamma_{K_{2m}}$ with essentially the same computational cost needed to get 
$\gamma_{K_{1}}$ and $\gamma_{K_{2}}$.
Using Pari/Gp \cite{PARI2021} we implemented this,  with a precision of 30 decimal digits,
for each odd prime $q\le 3000$
and $1\le m \le 6$; this required about 33 minutes of computing time.
For $q>3000$, the use of the Fast Fourier Transform algorithm
is mandatory, as explained in  
\cite{Languasco2021a, LanguascoR2021};
the accuracy of the latter procedure is commented on in \cite{LanguascoR2021}.

We did not perform such FFT computations for $m\ge 2$, since in these cases 
we would not be able to prove an analogue of Theorem \ref{main2}. To reach this goal, in fact, 
we should obtain the values of $S(m,q)$ for $q$ up to very large bounds, see
Remark \ref{q-bounds-remark}, 
which is currently infeasible because it
is much harder to compute $S(m,q)$, which is defined as in \eqref{gp-def}
but using $m\mid (q-1)/2$ and $g_p= f_p/(f_p,m)$. 
The prime sums involved were 
computed up to a certain bound $P$ and then we estimated
the remaining tails exploiting the summation functions of Pari/Gp
\cite{PARI2021}.  The slow decay ratio of some of the summands
in $S(m,q)$ prevents us from obtaining a very good accuracy. 
However, by choosing $P=10^8$ first, and then using $P=10^9$ or $P=10^{10}$
if necessary,  we were able to handle all the cases  $3\le q\le 3000$ with
$1\le m\le 6$ and $m\mid (q-1)/2$, 
with sufficient accuracy to determine the winner
in the ``Landau vs.~Ramanujan" problem; this required about a week of computation time.
For the cases in Tables \ref{tab:LvRnew1}, \ref{tab:LvRnew2}, \ref{tab:LvRnew3},
\ref{tab:LvRnew2decimals} and \ref{tab:LvRnew}, with the choice $P=10^8$ 
resulting in a final accuracy less than $6$ decimal digits, we repeated the 
computation using $P=10^9$ or $P=10^{10}$.
Some practical tricks were used to improve the actual running time of 
this part.
First, for a fixed odd prime $q$ we scanned the set of primes $p\le P$ just once
storing the partial results of each sum in the definition of $S(m,q)$ in a matrix having
a row for each requested $m$ (the largest possible set in our implementation
of this part is $1\le m\le 6$). Second, to have the sharpest possible estimate
for the ``tails'', for every $q$ we stored the set of $g_p$-values used in the
previous procedure so that the evaluated upper bound for such tails 
were based 
just on the effectively used $g_p$ and not over every divisor of $q-1$. The  
computations with $P=10^8$ for every odd prime up to $3000$ and $1\le m\le 6$ 
were performed on the Dell Optiplex machine already mentioned 
and required about $40$ hours of computing time.
The ones with $P\in\{10^9,10^{10}\}$ were performed on six machines of the
cluster of the Dipartimento di Matematica of the University of Padova;
in this case the total computing time amounted to $45$ days.

\subsection{Accelerated convergence formulae for $\gamma_{k,r}$}
For any $J\ge 2$ we rewrite \eqref{firstlog} as 
\begin{equation*}
\label{secondlog}
\frac{T'}{T}(s)=\frac{\zeta'}{\zeta}(s)+\frac{D'}{D}(s)
+\sum_{g_p=2}\frac{\log p}{p^s}
-
\sum_{j=2}^J (-1)^j\sum_{g_p=2}\frac{\log p}{p^{js}}
+
(-1)^J\sum_{g_p=2}\frac{\log p}{p^{Js}(p^s+1)}.
\end{equation*}
By Lemma \ref{logp}, sums of the form
$\sum_{p\equiv a\pmod*q}(\log p)p^{-js}$ can be
expressed in terms $L'/L(js,\chi)'s$ and sums of the
same type, but with $j$ replaced by $2j.$ The upshot
is that we can write the right-hand side in terms of
$L'/L(js,\chi)'s$ with $j\le J$ and with an 
error term of
the form $O(\sum_{p}(\log p)p^{-(J+1)s}).$ The 
same applies to the logarithmic derivative $D'/D(s).$
This reasoning suggests that we can express $T(s)$ 
itself in terms of $L'/L(js,\chi)'s$ with $j\le J$ and
a regular function for Re$(s)>1/(J+1)$. In the next
section we confirm this supposition. 
\subsubsection{Higher level $L$-factorability of $T(s)$}
\begin{Def}
Let $q$ be a fixed odd prime.
We say a Dirichlet-series $F(s)$ is 
\emph{$L$-factorable of level $\ell$} if there are integers
$j,e,e_{\chi,j_1}$ such that
\begin{equation}
\label{Lfac}
F(s)^j=\zeta(s)^eR(s)\prod_{j_1=1}^\ell\prod_{\chi}L(j_1s,\chi)^{e_{\chi,j_1}},
\end{equation}
with $R(s)$ a regular function for Re$(s)>1/(\ell+1)$ and
where $\chi$ runs over the non-principal characters modulo $q.$
We say that a set of primes $\mathcal P$ is $L$-factorable 
of level $\ell$ 
if $\prod_{p\in \mathcal P}(1-p^{-s})^{-1}$ is
$L$-factorable of level $\ell.$
\end{Def}
Notice that the product of two $L$-factorable functions
of level $\ell$ is $L$-factorable of level $\ell$ again.
It is a classical fact that 
the set of primes splitting completely in any prescribed
subfield of $\mathbb Q(\zeta_q)$ is $L$-factorable 
of level 1.
Thus the set of primes with $f_p=1$ is $L$-factorable 
of level 1.
The regular part $R(s)$ consists of Euler products of
the form $\prod_{f_p=j}(1-p^{-e_js})^{-1},$ 
with $e_j\ge 2.$ Each of these is $L$-factorable
of level $2,$ 
with the new regular part $R(s)$ consisting of 
Euler products of
the form $\prod_{f_p=j}(1-p^{-e_js})^{-1},$ 
with $e_j\ge 3.$
We conclude that for arbitrary 
$\ell\ge 1$ the set of primes that split completely in any
subfield of $\mathbb Q(\zeta_q)$ is $L$-factorable 
of level $\ell$.
Given $m$ dividing $q-1,$ the set of primes 
with $f_p$ dividing $m$ is also $L$-factorable
of level $\ell$, as
this is the set of primes that split completely in
$K_m.$ By inclusion-exclusion we then infer that
the set of primes $p$ with $f_p=m$ is $L$-factorable
of level $\ell$.
\begin{Prop} 
\label{prop:level}
Let $q$ be an odd prime and $k\ge 1$
an arbitrary integer.
Then $T(s):=\sum_{q\nmid \sigma_k(n)}n^{-s}$ is $L$-factorable
of level $\ell,$ with $\ell\ge 1$ arbitrary. 
\end{Prop}
\begin{proof}
The Euler product \eqref{EulerProductTsgeneral2} for $T(s)$ consists of 
Euler products
of the form $\prod_{g_p=m}(1-p^{-e_ms})^{-1},$ 
with $m$ running over the divisors
of $q-1$ and with $e_m\ge 1.$ Recalling that
$g_p=f_p/(f_p,r)$ with 
$r=(k,q-1),$ we see that the set of primes
with $g_p=m,$ is a union of sets of primes of the
form $f_p=m_i.$ Each of these prime sets is 
$L$-factorable of level $l$ and hence so is
$T(s).$
\end{proof}
As usual, let
$h=(q-1)/r.$ From \eqref{Tsh} and \eqref{Dedekindratio} it follows that $T(s)$ is $L$-factorable of level 1,
with $F(s)=T(s),$ $j=h$ and $e=h-1$ in \eqref{Lfac}.
\begin{Rem}
 Proposition \ref{prop:level} can also be proved using the theory developed in Ettahri et al.\,\cite{ERS} (communication 
 by Olivier Ramar\'e).
\end{Rem}
\subsection{Special cases}
\label{specialcases}
In certain special cases the convergence
can be improved.
We start by noting that
for $k\ge 1$ we have
$$\prod_{\kronecker{D}{p} = -1}\frac{1}{(1-p^{-ks})^2}=\frac{\zeta(ks)}{L(ks,\chi_D)}\prod_{p\mid D}(1-p^{-ks})\prod_{\kronecker{D}{p} = -1}\frac{1}{(1-p^{-2ks})}$$
and
$$\prod_{\kronecker{D}{p} = 1}\frac{1}{(1-p^{-ks})^2}=L(ks,\chi_D)\frac{\zeta(ks)}{\zeta(2ks)}\prod_{p\mid D}(1+p^{-ks})^{-1}\prod_{\kronecker{D}{p} = 1}\frac{1}{(1-p^{-2ks})}.$$
To see this we partition the primes $p$ according to the Legendre symbol $\kronecker{D}{p}$ and verify that, in each case, the Euler product factor at $p$ on the left-hand side equals that on the right-hand side.
\par By logarithmic differentiation  we obtain that, for $k\ge 2$,
\begin{equation}
\label{convergenceincrease}
\sum_{\kronecker{D}{p}= -1}\frac{\log p}{p^k-1}
=\sum_{\kronecker{D}{p} = -1}\frac{\log p}{p^{2k}-1}
+\frac{1}{2}\bigg(
\frac{L'}{L}(k,\chi_D) 
-
\frac{\zeta'}{\zeta}(k)
-\sum_{p\mid D}\frac{\log p}{p^k-1}\bigg)
\end{equation}
and
\begin{equation}
\label{convergenceincrease2}
\sum_{\kronecker{D}{p}= 1}\frac{\log p}{p^k-1}
=\sum_{\kronecker{D}{p} = 1}\frac{\log p}{p^{2k}-1}
-\frac{1}{2}\bigg(
\frac{L'}{L}(k,\chi_D) 
+ 
\frac{\zeta'}{\zeta}(k)
- 2  
\frac{\zeta'}{\zeta}(2k)
+\sum_{p\mid D}\frac{\log p}{p^k+1}\bigg).
\end{equation} 

Assume that  $r=(q-1)/2$ and 
$q\equiv 3\pmod*{4}$. In this case
the condition $g_p=2$ is equivalent
with $p^{(q-1)/2}\equiv \kronecker{p}{q} =-1$.
By quadratic reciprocity we have
$\kronecker{p}{q} = \kronecker{-q}{p}$. Using
\eqref{convergenceincrease} we conclude 
that 
$$\sum_{g_p=2}\frac{\log p}{p^2-1}
=\sum_{g_p=2}\frac{\log p}{p^4-1}+\frac{1}{2}\bigg(
\frac{L'}{L}(2,\chi_{-q})
-
\frac{\zeta'}{\zeta}(2)
-\frac{\log q}{q^2-1}\bigg).$$
Moreover, in this case, the condition $g_p=1$ is equivalent to $\kronecker{p}{q} =1$;
so, inserting formula \eqref{convergenceincrease2} for $k=q-1$ and $k=q$ into \eqref{Srq-def},
we can improve the convergence ratio of this sum too.
In fact, both formulae \eqref{convergenceincrease} and \eqref{convergenceincrease2} can be iterated several times.
Implementing this strategy we were able 
to compute $\gamma_{(q-1)/2,q}$,
as described in Section \ref{EK-quadratic}, for each odd prime $q\le 3000$
with an accuracy of $50$ decimal digits in less than $123$ seconds of computing time.

The previous argument requires to compute 
$L'/L(j,\chi)$, 
for $j\ge 2$.
To obtain such values we can use ($q$ odd prime, $\Re(s)>1$) that
\[
L(s, \chi) = q^{-s} \sum_{a=1}^{q-1} \chi(a)\zeta(s,a/q)
\quad\textrm{and}
\quad
L'(s, \chi) = -(\log q) L(s,\chi) + q^{-s} \sum_{a=1}^{q-1} \chi(a)\zeta^\prime(s,a/q),
\]
where $\zeta(s,x)$ is the Hurwitz zeta function and $\zeta^\prime(s,x): = \frac{\partial\zeta}{\partial s}(s,x)$, $\Re(s)>1$, $x>0$.
Hence
\begin{equation}
\label{L-hurwitz}  
\frac{L'}{L}(j,\chi)
=  - \log q +\frac{ \sum_{a=1}^{q-1} \chi(a)\zeta^\prime(j,a/q)}{ \sum_{a=1}^{q-1} \chi(a)\zeta(j,a/q)}.
\end{equation}

\subsection{An application of the convergence acceleration technique: the Shanks constant} 
In 1964, Shanks \cite{Shanks}  was the first to use \eqref{convergenceincrease} to study the
behavior of $B(x)$, the number of integers
less or equal to $x$ that are the sum of two squares.
We show now how this works 
and how to  improve some of the known results on this problem. Shanks obtained that
\[
B(x) = \frac{{\mathcal K}x}{\sqrt{\log x}}\bigg(1+ \frac{c}{\log x} + O\bigg( \frac{1}{\log^2 x}\bigg) \bigg),
\]
as $x \to \infty$, where
${\mathcal K}$ is the Landau-Ramanujan  
constant (see \eqref{LRconstant}) and
\begin{equation}
\label{shanks-coeff-def}
c= \frac12 +\frac{\log 2 }{4} - \frac{\gamma}{4} -  \frac14
\frac{L'}{L}(1,\chi_{-4}) 
+
\frac 12 \sum_{p\equiv 3 \pmod*4} \frac{\log p}{p^2-1},
\end{equation}
with $\chi_{-4}(\cdot) = \kronecker{-4}{\cdot}$ being the quadratic Dirichlet character modulo $4$. 
The associated Euler-Kronecker constant $\gamma_{SB}$ satisfies $\gamma_{SB}=1-2c$
by Theorem \ref{thm:multiplicativeset}. 
Iteratively using \eqref{convergenceincrease} $J_c\ge 1$ times,
we obtain that
\begin{equation}
\label{Shanks-boost}
\sum_{p\equiv 3 \pmod*4} \frac{\log p}{p^2-1}
=
\frac{1}{2} 
\sum_{j=1}^{J_c}\bigg( 
\frac{L'}{L}(2^j,\chi_{-4})
-
\frac{\zeta'}{\zeta}(2^j)
-\frac{\log 2}{2^{2^j}-1}\bigg)
+
\sum_{p\equiv 3 \pmod*4} \frac{\log p}{p^{2^{J_c+1}}-1},
\end{equation}
which, for $J_c=2$, gives  eq.~(18) of \cite{Shanks}.
Shanks wrote $b$ instead of ${\mathcal K},$ and obtained a very similar formula
whose truncated form can be written as follows:
\[
\sum_{p\equiv 3 \pmod*4}  \log  \Bigl( 1- \frac{1}{p^2} \Bigr)
=
-
\sum_{j=1}^{J_b}\frac{1}{2^{j}} 
\log \Bigl( \frac{\zeta(2^j)(1-2^{-2^j})}{L(2^j,\chi_{-4})} \Bigr) 
+
\frac{1}{{2^{J_b}}}\!\!\!\!
\sum_{p\equiv 3 \pmod*4}  \log  \Bigl( 1- \frac{1}{p^{2^{J_b+1}}} \Bigr),
\]
where  $J_b\ge 1$ is an integer.
A straightforward argument proves that
the last sum in \eqref{Shanks-boost} 
does not exceed $(\log 3)\cdot 4^{1-2^{J_c}}/3,$ 
so that in order to show that this term
is less that $10^{-\alpha}$, it is enough to choose 
\[
J_c > \frac{1}{\log 2} \log  \bigg(1+  \frac{\alpha  \log 10 + \log \log 3 - \log 3}{2\log 2}\bigg).
\]
 Similar remarks applies to $J_b$ too.
For example, for $\alpha =100$, it is enough to choose $J_b=J_c=8$.
Since $\chi_{-4}$ is an odd primitive character, we can write 
$L^\prime/L(1,\chi_{-4})$ 
in terms of the $\log \Gamma$-function and of the first $\chi$-Bernoulli number,
see, e.g., \cite[\S3]{Languasco2021a}. 
Straightforward computations give
\begin{equation*}
\label{quadlogder-at1} 
\frac{L'}{L}(1,\chi_{-4})
 =
 \gamma + 2 \log 2 + 3 \log \pi - 4 \log \Gamma\Big(\frac{1}{4}\Big),
\end{equation*}
and the needed Gamma-value can be obtained using
the Arithmetic-Geometric Mean (AGM) inequality, see, e.g., 
Borwein-Zucker \cite{BorweinZucker1992}.
The contribution of $L'/L(2^j,\chi_{-4})$
can be evaluated using
\eqref{L-hurwitz}, which in this case becomes
\begin{equation}
\label{quadlogder-at2^k} 
\frac{L'}{L}(2^j,\chi_{-4})
=
-2\log 2 + \frac{\zeta^\prime (2^j,1/4) - \zeta^\prime (2^j,3/4) }
{ \zeta(2^j,1/4)-\zeta(2^j,3/4)}.
\end{equation}
We remark that for $j=1$ the denominator in the previous equation is an integer multiple of the \emph{Catalan constant} $G,$ since it is well-known that $16 G= \zeta(2,1/4)-\zeta(2,3/4)$.

Inserting \eqref{Shanks-boost} and \eqref{quadlogder-at2^k} into \eqref{shanks-coeff-def}, 
we obtain an explicit formula that can be directly 
used in any mathematical software in which the Hurwitz zeta function is implemented.
Using Pari/Gp,  for instance, and choosing $J_c=8,$ we can obtain at least $100$ correct decimal digits of $c$
(and, in fact, also for ${\mathcal K}$) in about $38$ milliseconds; choosing $J_c=11$  we get  at least
$1000$ correct decimal digits in less than 
$4$ seconds of computation time; in about $383$ minutes, with $J_c=16$, 
we can get at least $31000$ correct decimal digits
(such computations were performed on the Dell Optiplex machine 
previously mentioned, using up to 12GB of RAM). 
In OEIS,  the Landau-Ramanujan 
constant  ${\mathcal K}$ appears as A064533,  
with about $125000$ digits available, while
the Shanks constant $c$ is mentioned as A227158, 
with about $5000$ digits available.

We finally remark that Ettahri, Ramar\'e and Surel \cite{ERS} further boosted the idea of Shanks and gave
it a more systematic setting using group theory.

\section{Outlook} 
\label{sec:Outlook}
\subsection{Generalizations}
The following result of 
Datskovsky and Guerzhoy \cite{DG}
shows that Ramanujan type congruences abound.
\begin{CThm}
\label{DatsG}
For any even integer $w\ge 12$ there exists a nonzero
cusp form $f=\sum a(n)q_1^n$ of weight $w$ with 
rational Fourier coefficients $a(n),$ so 
that for every $n\ge 1$ we have
$v_q(a(n)-\sigma_{w-1}(n))\ge 1,$ 
where  $q$ can be any prime divisor of 
the numerator of the
reduced fraction $\frac{B_w}{2w},$
$v_q$ is the $q$-adic valuation and $B_w$ denotes the $w$-th Bernoulli number.
\end{CThm}
In case $\dim S_w=1,$ it is easy to deduce
from this
that for $f$ we can take the unique cusp form of
weight $w$ normalized so that $a(1)=1.$ This
allows one to 
obtain the type (i) congruences satisfying
$w>q$ without a coprimality condition and
gives
an alternative proof of the second statement in
Proposition \ref{prop:belangrijk}.
\par For some congruence subgroups $\Gamma_0(N),$
Ramanujan-type congruences are known where the
relevant Fourier coefficients satisfy
$a(n)\equiv \sigma_k(n)\pmod*{q}$ for all $n$ 
coprime to $N,$ 
see, e.g., Kulle \cite{Kulle}.
The associated generating series
will be as $T(s)$ above, except for some possible modified
Euler product factors at primes $p$ dividing $N.$ These
factors can be easily logarithmically differentiated and
we can express the Euler-Kronecker constant as
$\gamma_{k,q}$ plus possibly a sum of terms involving
the primes $p$ dividing $N.$ Dummigan and 
Fretwell \cite{DF} gave a result similar 
to Classical Theorem \ref{DatsG} for $\Gamma_0(p),$
with $p$ prime.
\par The divisor sums arise as Fourier coefficients of
Eisenstein series. Over the years, many generalized Eisenstein
series have been considered, for example $E_w^{\psi,\xi},$
which involves two Dirichlet characters $\psi$ and $\xi$ 
(see Diamond and Shurman \cite[Thm.\,4.5.1]{DS}).
Its Fourier coefficients are of the form 
$\sum_{d\mid n}\psi(\frac{n}{d})\chi(d)d^{w-1}$ 
and can likely also be dealt with using our 
methods. 
The non-divisibility
asymptotics, in the special case where $\psi$ is the
principal character, were determined by
Scourfield \cite{Scourfield64,Scourfield72}. 
Here, if $\chi$ is a Dirichlet character modulo $N,$ then
the divisor sum  $\sum_{d\mid n}\chi(d)d^{w-1}$ is the
$n$-th Fourier coefficient of the Eisenstein series of
weight $w$ and character $\chi$ on $\Gamma_0(N),$
see, e.g., the book \cite[p.\,17]{123}.
\subsection{Regarding our conjectures}
One might hope that 
Conjectures \ref{Conj_r=3} and \ref{Conj_r=5} can
be proved under GRH. Indeed, the analysis of
the ``Landau vs.\,Ramanujan problem" using GRH 
(pioneered by Ihara \cite{Ismallnorm}) is 
technically far less demanding; for this, compare
Moree \cite{Mpreprint} (on GRH) with 
Ford et al.\,\cite{FordLucaMoree} (unconditional).
However, in our case, the bottleneck is represented by  the
behavior and slow decay rate of $S(3,q)$ and $S(5,q).$
\subsection{Some open questions}
\begin{itemize}
\item Solve the ``Landau vs. Ramanujan problem" for non-exceptional primes.
\item What are the optimal upper bounds
for the prime sums $S(m,q)$? 
How do they behave on average (with $m$ fixed)?
\item Consider the number $N(x)$ of pairs $(k,q)$ with
$1\le k,q\le x$ with $q$ prime 
for which Landau wins, that is, for which
$\gamma_{k,q}>1/2.$ Is it true that Landau wins
almost always in the sense that 
asymptotically $N(x)\sim x^2/\log x$? 
\item Given (any) $\epsilon>0,$ are there
 $k$ and $q$ such that
$|\gamma_{k,q}-1/2| < \epsilon$?
\item What is the average behavior of 
$\gamma_{k,q}$ for $q$ fixed?
\item How is $\gamma_{(q-1)/2,q}$ distributed as $q$ 
runs over the primes?
\end{itemize}

\section*{Acknowledgments} 
The authors are very thankful for the expert advice of
Bruce Berndt, Johan Bosman, Nikos Diamantis, Neil Dummigan, Pavel Guerzhoy, Bernhard 
Heim, Kamal Khuri-Makdisi, Ken Ono, Martin Raum and 
Sujeet Kumar Singh regarding the modular form aspects
of the paper, \and Florian Luca, 
Olivier Ramar\'e and Peter Stevenhagen on other aspects.
\par The paper was completed during a stay of the first author at the Max-Planck-Institut f\"ur Mathematik in Bonn. He is grateful for the inspiring atmosphere (even during pandemic times), the staff hospitality and the  excellent working conditions provided by the institute.
\par The harder part of the computations needed to  
establish Theorem \ref{main2}
were performed on the cluster of the Dipartimento di Matematica
``Tullio Levi-Civita'' of the University of 
Padova, see \url{http://computing.math.unipd.it/highpc}; 
the second
author is grateful for having 
had such computing facilities at his disposal.
\par The third author thanks Sir David Abrahams (Isaac Newton
Institute, Cambridge) for arranging
an opportunity for him to browse through the original  unpublished manuscript 
\cite{BerndtOno}. A magical feeling!

\end{document}